\documentclass[11pt,leqno]{article}

\usepackage{epsfig}

\usepackage[framemethod=tikz]{mdframed}
\usepackage{lipsum}

	\newmdenv[innerlinewidth=0.5pt, roundcorner=4pt,linecolor=outerspace,innerleftmargin=6pt,
	innerrightmargin=6pt,innertopmargin=6pt,innerbottommargin=6pt]{mybox}
	
	\usepackage{tcolorbox}  
	\newtcolorbox{mycolorbox}{colback=blue,colframe=outerspace}
	\newtcolorbox{mytbox}[1]{arc=8ex,lower separated=false,coltitle=black,colback=white,colframe=gray!20!white,boxrule=1pt,leftrule=1mm,rightrule=1mm,title=#1}

\usepackage{hyperref}

\usepackage{amssymb}
\usepackage{amsmath}

\newenvironment{proof}{\paragraph{\underline{Proof} :}}{\hfill$\square$}

\newenvironment{proofth}[1]{\paragraph{\underline{#1}} :}{\hfill$\square$}

\usepackage{mathtools}
\mathtoolsset{showonlyrefs}

\newtheorem{theorem}{Theorem}
\newtheorem{proposition}{Proposition}

\newtheorem{lemma}{Lemma}
\newtheorem{remark}{Remark}
\newtheorem{corollary}{Corollary}

\newcommand{\A}{\mathcal{A}}
\newcommand{\B}{\mathcal{B}}
\newcommand{\D}{\mathcal{D}}

\newcommand{\I}{\mathcal{I}}

\newcommand{\E}{\mathcal{E}}

\newcommand{\N}{\mathcal{N}}

\newcommand\CC{\hbox{C\kern -.58em {\raise .54ex \hbox
			{$\scriptscriptstyle |$}}
		\kern-.55em {\raise .53ex \hbox{$\scriptscriptstyle |$}} }}
\newcommand\qd{\hfill$\sqcap\kern-8.0pt\hbox{$\sqcup$}$}
\newcommand\NN{\hbox{I\kern-.2em\hbox{N}}}
\newcommand\nn{\hbox{I\kern-.2em\hbox{N}}}
\newcommand\RR{I\!\!R}
\newcommand\sRR{{\sl \hbox{I\kern-.2em\hbox{R}}}}
\newcommand\QQ{\hbox{I\kern-.53em\hbox{p}}}

\newcommand{\argmin}{\hbox{argmin}}     

\everymath{\displaystyle}

\DeclareMathOperator*{\argmax}{arg\,max}

\usepackage{stackengine}
\usepackage{scalerel}
\usepackage{xcolor,amssymb}

\title{Cross-Diffusion Theory for Overcrowding Dispersal \\  in Interacting Species System }
\author{Noureddine Igbida  \thanks{Institut de recherche XLIM, UMR-CNRS 7252, Faculté des Sciences et Techniques, Université de Limoges, France. Email:  noureddine.igbida@unilim.fr}}

\date{\today}

\begin{document}
 \maketitle

 \begin{abstract}

This work introduces  a new class of cross-diffusion systems for studying   overcrowding dispersal of two species.  The approach, based on proximal minimization energy through a minimum flow process, offers a potential generalization of existing segregation models. Unlike prior methods using PDEs or $W_2$-Wasserstein flows, it establishes a well-posed PDE  framework for capturing the interplay between diffusion and concentration gradients. This framework has the potential to significantly improve our understanding of how cross-diffusion shapes spatial patterns, coexistence, and overall distribution of multiple species. Notably, for homogeneous cases, the approach definitely leads   to a well-defined PDE grounded in a new general $H^{-1}$-theory specifically developed for overcrowding dispersal. This theory provides a robust foundation for further analysis.

 \end{abstract}

\section{Introduction and preliminaries}
 
  In this study, we focus on a specific type of cross-diffusion system involving overcrowding dispersal of two species. Overcrowding dispersal in interacting species refers to the phenomenon where two or more species, when encountering high population density in a shared environment, exhibit behaviors that encourage them to spread out and reduce their local density. 
  Cross-diffusion systems have gained significant attention in mathematical modeling as they provide a more realistic description of interactions and dynamics between different species in various scientific fields, including ecology, chemistry, and biology. In this context, a cross-diffusion system captures the interplay between multiple species by considering the influence of one species on the diffusion rates of another.  It is known that  mathematical analysis and numerical simulations of these kind of problems  exhibit complex dynamics and patterns  which may be handled using  cross-diffusion system. The cross-diffusion terms in the system take into account the fact that the diffusion rates of each species are influenced not only by their own concentration gradient but also by the concentration gradient of the other species.

  Let us denote by $\rho_k(t,x)$ the concentration of each specie $k=1,2,$ l  at time $t\geq 0$ at position $x\in \Omega$, where  $\Omega\subset \RR^N $ is a given open bounded domain. Here  $N\geq 1$, which represents space dimension,  may be considered equals to $1,2$ or $3$ for practical situations.     A particular example of cross-diffusion  system we shall discuss in this paper for the couple of interacting densities $(\rho_1,\rho_2)$   is given by the balances law  for each $\rho_k,$ $k=1,2,$ 
  $$  	\partial_t \rho_k +\nabla \cdot J_k[\rho_1,\rho_2]= f_k, \quad \hbox{ in } Q:(0,T)\times \Omega,$$
   where $T>0$ is a given horizon time, $J_k[\rho_1,\rho_2]$ is the mass-flux   and $f_k$ represents  reaction term or simply  species supply. 
  To model the dispersal phenomena Gurtin et al. propose in \cite{Gurtin84} to consider the intuitive constitutive relation for each    $J_k$ given by   
 $$J_k[\rho_1,\rho_2]= \rho_k \: \upsilon _k[\rho_1,\rho_2]  ,$$
relating the   mass-flux  $J_k[\rho_1,\rho_2]  $  to  the dispersal velocity $\upsilon_k,$  which measures, among other things, the desire of individuals for each specie $k=1,2,$ to migrate.   Moreover, for the dispersal phenomena, they propose  to consider, for each $k=1,2,$ 
 $$ \upsilon_k[\rho_1,\rho_2] = -  \sigma_k  \: \nabla (\rho_1+\rho_2) ,$$
  where $0<\sigma_k$ is assumed to be  connected to the \emph{dispersivity}, and may depend on $\rho_k$ and $\rho_1+\rho_2,$ so that each specie    disperses (locally) toward lower values of total population.    
 Then, the underlying PDEs take the form of cross-diffusion system 
\begin{equation}\label{crossgen1}
	 \left\{ \begin{array}{ll}
 	\partial_t \rho_1-  \nabla\cdot(  \sigma_1 \:  \rho_1\:    \nabla (\rho_1+\rho_2)) = f_1\\   \\ 
 	\partial_t \rho_2-  \nabla\cdot(  \sigma_2 \:   \rho_2\:   \nabla (\rho_1+\rho_2))= f_2
 \end{array}\right. \quad \hbox{ in } Q . 
\end{equation}
Since the seminal paper by  \cite{Gurtin84},   the cross-dffusion system \eqref{crossgen1}   was proposed in a series of works with  $\sigma_k=1$ ; i.e 
\begin{equation}\label{crossgen2}
	\left\{ \begin{array}{ll}
		\partial_t \rho_1-  \nabla\cdot(   \rho_1  \:   \nabla (\rho_1+\rho_2)) = f_1\\   \\ 
		\partial_t \rho_2-  \nabla\cdot(   \rho_2   \:  \nabla (\rho_1+\rho_2))= f_2.
	\end{array}\right.   \quad \hbox{ in } Q 
\end{equation} subject to homogeneous Neumann boundary condition  to study segregation theory where one species dominates (see \cite{Hilhorst87,Hilhorst2012,SantambrogioCD,Perthame12}).  Following  the works in \cite{BDFS,SantambrogioCD,KM} in the   framework of gradient flow in Wasserstein space 
it is well-established that systems like   \eqref{crossgen1}  are closely linked to a minimization process of an internal energy functional acting on the total mass $\rho_2+\rho_2,$ i.e. $\int\tilde \beta(\rho_1+\rho_2)$. 
 However, existing results are often partial and heavily influenced by "stable" configurations (as defined in \cite{KM}) that avoid mixing between the species. This limitation restricts the study of these type of systems to specific spatial dimensions and  source/reaction term behaviors $f_k.$  The case where $f_k$ arises from a  drift   is concerned  as well, since the   study  become particularly complex  even in one dimensional, as explored in   \cite{KM}.    
  
Actually, while the system \eqref{crossgen1} initially appears to exhibit fast-wave propagation (hyperbolic behavior) for each individual density $\rho_k,$   it surprisingly becomes diffusion-dominated (parabolic behavior) when considering the total cell density $\rho_1+\rho_2$.    Interestingly, stable configuration (the segregation regime ; i.e. $\rho_1\: \rho_2=0$)  suggests that the system might still retain some diffusion-like behavior even for individual densities.      Nevertheless, this \emph{strong} formulation (hyperbolic) in  \eqref{crossgen2} essentially implies that the dynamics of each population are solely determined by the combined dynamics of both populations. This  seems to  imposes  a somewhat inflexible constraint on the individual dynamic of each population. The situation appears to be better from theoretical point of view  when the reaction terms $f_k$ make possible segregation phenomena (see \cite{Hilhorst87,Hilhorst2012,SantambrogioCD,Perthame12}).  The formulation is     suitable  also  when coming moreover with a diffusion terms for each populations   allowing to consider  in some sens individual  Brownian motion for each populations (cf. \cite{Laborde} for the study of crowd motion).   In this case the mass-fluxes turns into  
  $$J_k[\rho_1,\rho_2] = -  \sigma_k  \: \nabla (\rho_1+\rho_2) - \nabla \rho_k, $$ 
  so that a diffusion terms $-\Delta  \rho_k$  associated with each density $\rho_k$ enter in the system and make it more flexible from the theoretical and numerical points of view. 
  
  \medskip 
  Our approach in this paper is   different,   we follow gradient flow approach  in Wasserstein spaces, by using   the system's internal energy expressed through  a cost functional acting on the total density ; $\int_\Omega  \beta(\rho_1+\rho_2).$ Here $\beta$ is assumed to be a convex and lower semi-continuous function, which may vary spatially.   We employ a process based on the steepest descent algorithm. However, to measure the transition work, we utilize the minimal flow approach we introduced in \cite{IgGF} for single-species, as an alternative of Wasserstein distance.  Roughly speaking, for a given initial condition $(\rho_{01},\rho_{02})$  and a fixed time step $\tau>0,$ we set  $(\rho^1_0,\rho_2^0)= (\rho_{01},\rho_{02}),$   and then recursively define 
 \begin{equation}\label{Prox0}
 	  (\rho_1^{i},\rho_2^{i}):= \argmin_{(\rho_1,\rho_2)} \left\{ \frac{1}{2\tau } (\I (\rho_1,\rho_1^{i-1}) + \I (\rho_2,\rho_2^{i-1}) +  \int_\Omega \beta(\rho_1+\rho_2) \: dx  \right\} ,\quad i=1,...n, 
 \end{equation}
  where $\I (\rho_k,\rho_k^{i-1}) $ is given by the minimum flow problem associated with the densities $\rho_k$ and $\rho_k^{i-1}$ (the exact definition in Section \ref{SectionProx}).  Remember that  this minimal flow approach is well-suited for incorporating diffusion processes into dynamical systems which  naturally tend towards minimizing their internal energy (cf. \cite{IgGF}).  This framework allows us here to  recover the   concept of dispersal potentials, $\eta_k,$ which capture the individual's desire to migrate based on the local densities of both species and the total density $\rho_1+\rho_2.$    Basically, 
  we  encode the dispersal behavior using specific parameters, $\eta _k,$ to rewrite    the mass flux for each species   directly  as:
  $$J_k[\rho_1,\rho_2] = - \sigma_k  \:  \nabla   \eta _k ,\quad  \hbox{ for each }k=1,2.$$  
   This formulation leads to parabolic equations governing the dynamics of the species densities,  $\rho_k$:
   $$ \partial_t \rho_k -\nabla \cdot(\sigma_k\:  \nabla \eta_k ) = f_k , \quad   \text{ for } k=1,2. $$
    By delving deeper into the optimization process for the system's internal energy, we uncover a crucial link between the dispersal potentials, $\eta_k,$  the individual densities, $\rho_k,$    the total density, $\rho_1+\rho_2,$ and the spatial   
    rearrangements  of the densities.    This connection leads  into significant equations, providing valuable tools for analyzing the dynamic behavior of the system.  For instance, in the case where $\beta$ is differentiable, this link may be written in the following form 
     \begin{equation} \label{Bcondition1}
    	\eta_k -  \beta'(\rho_1+\rho_2) \in \partial I\!\! I_{[0,\infty)} (\rho_k) \quad  \hbox{ for each }k=1,2,    \end{equation}
    	      which is equivalent to say    (see Remark \ref{Rmqequivalentform} for different related equivalent formulations)
      \begin{equation} \label{Bcondition0}
    	\left\{ \begin{array}{ll}
    		\eta_k=   \beta'(\rho_1+\rho_2) \quad &\hbox{ in }[\rho_k>0]\\  \\    \eta_k\leq     \beta'(\rho_1+\rho_2)   &\hbox{ in }[\rho_k=0].
    	\end{array}\right. 
    \end{equation}
    In the context of the dispersal overcrowding example \eqref{crossgen2}, this novel approach introduces a new model to capture the dynamics of this phenomenon. The method utilizes a parabolic function linked to a minimum internal energy process acting on the total density, which in this case is squared; specifically, this integral is given by : $\frac{1}{2}\int_\Omega (\rho_1+\rho_2)^2.$ The approach differs from traditional methods rooted in \cite{Gurtin84} which rely on phenomenological processes to model mass fluxes. The key distinction lies in how we handle the mass flux, $J_k[\rho_1,\rho_2],$  for each species $(k=1,2).$  We propose a different treatment within their active zones and outside them. While the proposed system enables to cover the system \eqref{crossgen2}  in the segregated regions  (i.e.  $[\rho_1>0\ \&\: \rho_2=0]$ and $[\rho_1=0\ \&\: \rho_2>0]$), it fundamentally changes in the regions where neither species are present (i.e.,  $[\rho_1>0\ \&\: \rho_2>0]$). More precisely, we propose reformulating the fluxes in  \eqref{crossgen2} as follows 
   \begin{equation}  
   	 \left.  \begin{array}{ll}
   	J_k[\rho_1,\rho_2] =-(\rho_1+\rho_2)    \: 	\nabla (\rho_1+\rho_2)  - \nabla \tilde \eta_k  \\  \\   
   	\tilde \eta_k \in \partial I\!\! I_{[0,\infty)}(\rho_k)
   	\end{array}\right\}  \quad \hbox{ for }k=1,2. 
	   \end{equation}  
 Here $\tilde \eta_k$ is an unknown potential   to be determined also.  From mathematical point of view it can be viewed  as a Lagrange multiplier associated with the positive sign constraint of the density $\rho_k.$ From a physical perspective, it could be interpreted as a corrective potential which prevents the system from collapsing into an \emph{unphysical} two-phase problem (i.e., changing sign solutions). This  offers valuable significant insights into the  behavior of the system. Notably, this allows us to address well-posedness questions, such as existence and uniqueness of solutions for more general situations (see Section \ref{Sexistence}).

   \medskip 
 Additionally, although the hyperbolic formulation in equation \eqref{crossgen2} might seem appropriate for Neumann boundary conditions (conservative case), it becomes less suitable for scenarios involving dynamic behavior with Dirichlet boundary conditions, such as collective motions with distinct  ''evacuation exits''.   This is because the hyperbolic nature suggests rapid wave propagation, which might not accurately capture the dynamics of controlled outflow through a specific location like an exit.   As we will explore further, among others this presents an intriguing challenge that our approach seeks to address.     Actually, following the approach \cite{IgGF}, by incorporating the boundary condition in an adequate way in the  internal energy (see Section \ref{SectionProx}), we suggest to handle the associate boundary value problem through the general mixed  boundary conditions 
   $$\left\{ \begin{array}{ll} 
J_k[\rho_1,\rho_2]\cdot \nu = \pi_k \quad &     \hbox{ on } \Gamma_{N_k} \\  \\    
 \eta_k=g_k \quad &     \hbox{ on } \Gamma_{N_k}    \end{array}\right. 
  \quad \text{ for } k=1,2,$$  
 where, for each $k=1,2,$ $\Gamma_{N_k}$ and $\Gamma_{D_k}$ constitute partitions of the boundary $\partial\Omega,$    $\pi_k$ represents a given normal flux for each specie $k$   and $g_k$ is a given charge related to  Dirichlet boundary section (an exit for instance).

  \medskip 
 Despite the variety of   boundary conditions, the framework  enables also  to integrate  further analysis in the  study of overcrowding dispersal dynamics.  Notably, it allows to investigate the general scenario where each specie   dynamic  is governed together by their  intrinsic  overcrowding dispersal forces and external actions  acting  like transport/convection phenomena. 
 To this aim,  we  introduce the cross-diffusion system 
 \begin{equation}\label{newcd0}
 	\partial_t 	\rho _k- \nabla\cdot ( \sigma_k\: \nabla\eta _k+ \rho_k V_k)=  f _k  ,  \quad  \hbox{ for }k=1,2,\quad \hbox{ in }Q:=(0,T)\times \Omega,
 \end{equation} 
where  we require that the dispersal potentials, $\eta _k,$  depend on the local densities of both species ($\rho_1$ and $\rho_2$), as well as the total density $\rho_1+\rho_2,$  following the specified  treatment  \eqref{Bcondition1}, or equivalently, any of the formulations presented in Remark~\ref{Rmqequivalentform}. Here $V_k$ is a given drift for each specie ($k=1,2$) assumed to be in $L^\infty(Q).$
 
 \medskip
Beside the freedom the new system  demonstrates for well-posedness   questions   for describing overcrowding dispersal dynamics, the approach  establishes that in the homogeneous case ; i.e. $\pi_k=0$ and $g_k=0$, the underlying PDE falls within  a  new general $H^{-1}$-gradient flow theory   specifically developed for overcrowding dispersal. In fact, the transition works $\I(\rho_k,\rho_k^{i-1})$ in \eqref{Prox0} we consider coincides with the square of $H^{-1}$-norm  of $\rho_k-\rho_k^{i-1}$  in this specific case (see Section \ref{SH-1}).  This theoretical framework provides a robust foundation for further analysis of the system's behavior like uniqueness and large time behavior.   

 It is meaningful to mention that while  gradient flow theories in $H^{-1}$ and $W_2$-Wasserstein space share   many conceptual similarities  for single-species overcrowding, existing studies  for multi-species system  using $W_2$-Wasserstein for interacting species haven't yielded a valuable PDE framework for well-posedness.   Additionally, for specific  scenarios like segregated densities, the $W_2$-Wasserstein approach seems to align more closely with hyperbolic à la  Gurtin et al. models than the parabolic cross-diffusion system proposed here.  This suggests a need for further investigation into the connections between these approaches for cross-diffusion systems.
 
  The most related work appears to be by \cite{KM}, who explore  degenerate cross-diffusion models with different drift velocities within a $W_2$-Wasserstein framework. In this work, the authors  consider the case where the energy is a power $m$ of the total density $\rho_1+\rho_2$ with $m$ possibly equal to infinity.   While their approach shares some similarities with ours at the discrete-time  level, the complexity of handling density mixing and strong restrictions both on the drift and the space dimension limit   the applicability of their continuum limit.  This highlights the potential of our proposed approach for broader applicability in studying overcrowding dispersal dynamics.

  \medskip
 The plan of the paper is the following: in Section \ref{Smain}, we gather some notation and the assumptions and state the main results. We present the existence of a  weak solution for the problem \eqref{PDE1}.  Then, for the homogeneous case and vanishing transport; i.e. $V\equiv 0,$ we show how one can connect the problem \eqref{PDE1} to a well established $H^{-1}-$gradient flow dynamic.  In Section \ref{Smodeling},  we show how to establish  the model by using minimum flow steepest gradient descent algorithm. In Section \ref{Sexistence}, we prove the existence of a weak solution by employing steepest gradient descent algorithm and passing to the limit in the approximate solution.  Finally, in  Section \ref{SH-1}, we provide the proof of the connection of the approach with $H^{-1}-$gradient flow in   the homogeneous case.

 \section{Assumptions and main results}\label{Smain}

 Throughout the paper, we assume that $\Omega$ is bounded regular domain with Lipschitz boundary $\partial \Omega,$ which can be decomposed as follows 
\begin{equation}\label{boundarydecomp}
	 \partial \Omega=\underbrace{\Gamma_{D_1}\cup \Gamma_{D_2} }_{\Gamma_D}  \cup\underbrace{  \Gamma_{N_1} \cup \Gamma_{N_2} }_{\Gamma_N}    , 
\end{equation}  with 
  $$ \mathcal H^{N-1}(\Gamma_{D_k})>0\quad  \hbox{ and }\quad \overline \Gamma_{D_k}\cap \overline \Gamma_{N_k}=\emptyset,\quad \hbox{ for each } k=1,2.$$  
 We consider a   function  $\beta\: :\:  \Omega\times \RR  \to \RR $  a carathéodory application satisfying the following assumptions: 
   \begin{itemize} 
   	
   	\item [(H1)]     $\beta(x,0)=0$ a.e. $x\in \Omega$  and, for a.e. $x\in \Omega$,  the application 
   	$$ r \in \RR \to \beta(x,r )   \hbox{ is lower semi-continuous  (l.s.c.) and convex} $$

   	\item [(H2)]    there exists $C,\: M>0$ such that 
   	$$C\: (\vert r\vert -M)^{+2 }\leq \beta(x,r),\quad \hbox{ for any }r\in \RR .$$   
   	
   \end{itemize}  
  Moreover, we fix   $\sigma=(\sigma_1,\sigma_2)\in [L^\infty(\Omega)]^2$ such that 
\begin{equation}\label{minsigma}
	0<\min_{x\in \overline \Omega} \min(\sigma_1(x),\sigma_2(x) ) .  
\end{equation}

\bigskip   
 In line with the preliminary discussion in the introduction, we will now formally introduce and analyze the cross-diffusion system : 
     \begin{equation} \label{CDpde}
    	\left\{  \begin{array}{ll} 
   			\partial_t 	\rho _k- \nabla\cdot ( \sigma_k\: \nabla\eta _k+ \rho_k V_k)=  f _k  ,   \\  \\ 
\eta_k -\tilde \eta    \in  \partial  I\!\! I_{[0,\infty)}(\rho_k),\: \tilde  \eta \in \partial \beta(x,  \rho_1+\rho_2)   
   		\end{array}\right\} 	\quad  \hbox{ in } Q,    \quad  \hbox{ for }k=1,2 .
   \end{equation}   
Here,  $V_k$ represents a vector-valued field, and $f_k$ is a source term, assumed to be provided for both species. We consider system \eqref{CDpde} with mixed boundary conditions of the following type: 
     \begin{equation} \label{BC0}	\left\{\begin{array}{ll} 
  		(\sigma_k\:\nabla\eta_k + \rho_k V_k )\cdot \nu  =  \tilde \pi_k   & \hbox{ on }\Sigma_{N _k}\\  \\
  		\eta_k  =  g _k ,\quad    & \hbox{ on }\Sigma_{D _k},
  	\end{array}
  	\right. 	 
  \end{equation}    
 where  $g_k$ and $\tilde \pi_k$ are given boundary data, for $k=1,2,$  we precise later.      We will provide a detailed motivation  for this system in Section  \ref{Smodeling}.

\bigskip 
To begin with, for  any $1\leq p\leq \infty$ and   $\tilde\Gamma \subset \partial\Omega $ given,   let us remind here some  notations and functional sets we use throughout the paper.  
 \begin{itemize}
 	 	\item[-]  
 	$L^{ p}(\Omega)$   denotes the usual Lebesgue  spaces endowed with their  natural norms.   Unless otherwise stated, and without abuse of notation, we will use $\int_\Omega  h$  to denote the Lebesgue integral over $\Omega,$ i.e. 
 	$$\int_\Omega  h := \int_\Omega h(x)\: dx, \quad \hbox{ for any }h\in \L^1(\Omega). $$
 	Moreover,     $L^p_+(\Omega)$ denotes the set of nonegative     functions of $L^p(\Omega).$  
 
 	 	 	\item[-]  
 	$H^1(\Omega)$ denotes the usual Sobolev space endowed with its natural norm 
 	$$ \Vert z\Vert_{H^1(\Omega)} =  \left( \int_\Omega  \vert z\vert^2 +\int_\Omega  \vert \nabla z\vert^2 \right) ^{1/2}.$$
 	
 	 	 	\item[-]   $H^1_{\tilde\Gamma} (\Omega) $ denotes  the closure, in $H^1(\Omega)$,  of $\mathcal C^1(\overline \Omega)$ functions which are null on $\tilde\Gamma.$   In particular, for any 
 	$z\in H^1_{\tilde\Gamma} (\Omega),$  $\gamma(z)=0$ on $\tilde\Gamma,$  where   $\gamma $  denotes the usual trace application defined from $H^1(\Omega)\to L^2(\partial\Omega).$  
 	
 	 	 	\item[-]      $ H^{-1}_{\tilde\Gamma }(\Omega)$ denotes the dual space  of $H^1_{\tilde\Gamma }(\Omega),$ $(H^1_{\tilde\Gamma }(\Omega)'.$ We know that (see for instance \cite{IgGF}) ,  for any $f\in   H^{-1}_{\tilde\Gamma }(\Omega),$ there exists a couple  $(f_0, \overline f)\in L^2(\Omega)\times L^2(\Omega) ^N,$ such that  
 	\begin{equation}\label{decompositionH}
 		\langle f,\xi\rangle_{H^{-1}_{\tilde\Gamma }(\Omega),H^1_{\tilde\Gamma }(\Omega)}  = \int_\Omega f_0\:\xi  -   \int_\Omega \overline f\cdot \nabla \xi  ,\quad \hbox{ for any }\xi\in 
 		H^1_{\Gamma_D}(\Omega). 
 	\end{equation}
 	Without abusing, we will use the couple $(f_0,\overline f)$ to identify    $f\in 	 H^{-1}_{\tilde\Gamma }(\Omega)$ and denote the duality bracket by $\langle .,.\rangle_{\Omega} $ ; i.e.  
 	$$\langle f,\xi\rangle_{\Omega}:= \langle f,\xi\rangle_{H^{-1}_{\tilde\Gamma }(\Omega),H^1_{\tilde\Gamma }(\Omega)},\quad \forall \xi\in H^1_{\tilde\Gamma }(\Omega).$$ 
 	Notice   that the couple $(f_0,\overline f)$ is not unique.  However, the resulting dual bracket does not depend on this specific  choice.

 	 	 	\item[-]      $H^{1/2}(\tilde\Gamma)$ denotes the usual   space given by $\gamma_{/\tilde\Gamma}(H^1(\Omega)),$ where $ \gamma_{/\tilde\Gamma}$ is the trace application  restricted to $\tilde\Gamma.$  We  need to define moreover   
 	$$  H^{1/2}_{00}(\tilde\Gamma):=\Big\{  \kappa\in L^2(\tilde\Gamma)\: :\:  \exists \: \tilde \kappa\in H^1 (\Omega),\: \gamma(\tilde \kappa)=\kappa \hbox{ on }  \tilde\Gamma  \hbox{ and } \gamma(\tilde \kappa)=0 \hbox{ on }\partial \Omega \setminus \tilde \Gamma  \Big\}.$$
 That is  $f=f_0+\nabla \cdot \overline f $ in $\D'(\overline \Omega\setminus\tilde  \Gamma).$
 	Thanks to \eqref{boundarydecomp}, taking $\tilde \Gamma=\Gamma_{D_k}$ or $\tilde \Gamma=  \Gamma_{N_k},$ for $k=1,2,$  the space  $H^{1/2}_{00}(\tilde\Gamma)$  coincides with the set of     functions belonging to   $H^{1/2}(\tilde\Gamma)$   and vanishing  on the   boundary  $\partial \Omega\setminus \tilde \Gamma.$  Remember that this not automatically true for any $\tilde\Gamma \subset \partial\Omega $.   

 	 	 	\item[-]   $ H^{-1/2}(\tilde\Gamma)$ denotes the usual   (topological) dual space of   $H^{1/2}_{00} ( \tilde\Gamma),$ $\left( H^{1/2}_{00}( \tilde\Gamma) \right)'  .$   We denote the duality bracket   simply by the    formal expression $\langle \pi ,\kappa\rangle_{\tilde \Gamma}$ ; i.e. \begin{equation} \label{notation1}
 		\langle \pi ,\kappa\rangle_{\tilde \Gamma} :=  \langle\pi,\kappa \rangle_{ H^{-1/2}( \tilde\Gamma),  H^{1/2}_{00}( \tilde\Gamma)} ,\quad \forall \kappa \in H^{1/2}_{00}( \tilde\Gamma) \hbox{ and }\pi \in H^{-1/2}(\tilde\Gamma)   .
 	\end{equation}
 	 
 	 	 	\item[-]   	We can define simultaneously the space  
 	$$ H_{div}(\Omega):=\Big\{  \upsilon\in L^2(\Omega)^N \: :\:   \nabla \cdot \upsilon  \in L^2(\Omega)
 	\Big\} ,$$
 	where $\nabla \cdot \upsilon$ is taken in $\mathcal D'(\Omega).$    

 	 	 	\item[-]   For any  $\upsilon\in H_{div}(\Omega),$   the normal trace  $\upsilon\cdot \nu$ is well defined on $\tilde\Gamma$ by duality.   More precisely,   $\upsilon\cdot \nu\in H^{-1/2}(\tilde\Gamma) $,    and  
 	\begin{equation}\label{tracexpression}
 		\langle \upsilon \cdot \nu ,\kappa  \rangle_{  \tilde\Gamma}  =\int_\Omega\upsilon \cdot \nabla \tilde \kappa  \: dx +  \int_\Omega \tilde \kappa \: \nabla \cdot \upsilon\: dx, 
 	\end{equation}  
 	for any $\kappa\in  H^{1/2}(\tilde\Gamma)$ and   $\tilde \kappa \in  H^1 (\Omega)$  such that  $\tilde \kappa =\kappa$ on $\tilde\Gamma  , $    and $\tilde \kappa =0$ on  $ \partial\Omega\setminus \tilde\Gamma .$

 	When dealing with the case where $\upsilon \in L^{p'}(\Omega)$ and $-\nabla \cdot \upsilon =f \in 	 H^{-1}_{\tilde\Gamma }(\Omega),$  defining the trace of $\upsilon$ on the boundary can be tricky. In general, it might not have a well-defined meaning.  Around this, the trace $(\upsilon+\overline f)\cdot \nu$ is well defined on $\partial \Omega,$ and does not depend on   the specific choice $\overline f$   in the decomposition $f=f_0+\nabla \cdot \overline f$  given by \eqref{decompositionH}.     	Therefore, when working with mixed boundary conditions like equation (\ref{BC0}), it becomes crucial to consider the individual boundary trace of  each  $f_k$ on $\Gamma_{N_k}.$ 

 \end{itemize}

 \subsection{Existence result}\label{sectionmainresultexist}
 
 To treat \eqref{CDpde} for general source terms $f_k$ in the dual space $ H^{-1}_{\Gamma_{D_k}}(\Omega)$, $k=1,2,$ we consider the following cross-diffusion system  
  \begin{equation}
 	\label{PDEevol}
 	\left\{\begin{array}{ll} 
 		\left.  \begin{array}{ll} 
 			\partial_t 	\rho _k- \nabla\cdot ( \sigma_k\: \nabla\eta _k+ \rho_k V_k+\overline f_k)=  f _{0k}  ,  \quad  \\ 
 			\\   \eta_k     \in \partial \beta(x,  \rho_1+\rho_2)  + \partial  I\!\! I_{[0,\infty)}(\rho_k)
 		\end{array}\right\} 	\quad & \hbox{ in } Q, \\  \\
 		(\sigma_k\: \nabla\eta_k + \rho_k V_k +\overline f_k)\cdot \nu  =  \pi_k   & \hbox{ on }\Sigma_{N _k}\\  \\
 		\eta_k  =  g _k ,\quad    & \hbox{ on }\Sigma_{D _k},\\  \\ 
 		\rho_k(0)=\rho_{0k}  & \hbox{ in } \Omega, 
 	\end{array}
 	\right\}  	   \hbox{ for }k=1,2 
 \end{equation}   
%
 where 
 \begin{itemize}
 	\item $(\pi_1,\pi_2)=\pi\in  H^{-1/2}(\Gamma_{N_1}) \times H^{-1/2}(\Gamma_{N_2})  $  
 	\item  $ (g_1,g_2)=g\in  H^{1/2}_{00}(\Gamma_{D_1}) \times H^{1/2}_{00}(\Gamma_{D_2})  $
 	 	\item $(f_1,f_2)=f\in L^2(0,T;  H^{-1}_{ \Gamma_{D_1} }(\Omega)\times H^{-1}_{ \Gamma_{D_2} }(\Omega)) $, where each $f_k$ is identified with the couple     $ (f_{0k},\overline f_k)   \in L^{2}(Q)\times L^{2}(Q) ^N,$ for each $k=1,2,$  as given by \eqref{decompositionH}.
 	\item $ (V_1,V_2)=V\in L^{\infty}(Q)^N\times L^{\infty}(Q) ^N $ 
 	\item  $  (\rho_{01},\rho_{02}) =\rho_0 \in L^{2}(Q)\times L^{2}(Q), $ assumed to be such that 
 	\begin{equation}\label{Condrho0}
 		\int_\Omega \beta(.,\rho_{01}+\rho_{02})<\infty.
 	\end{equation}    
 \end{itemize}
 
 \bigskip 
  Our first main result concerns existence of weak solution to \eqref{PDEevol}.

 \begin{theorem} \label{texistuniq} Under the previous assumptions,   the system of PDE  \eqref{PDEevol}  
 	has a solution $(\rho,\eta)$ in  the sense that,        for each $k=1,2,$  $\rho_k \in L^2_+(Q) \cap  W^{1,2}(0,T;H^{-1}_{\Gamma_{D_k}}(\Omega)),$  $\rho_{k}(0)=\rho_{0k},$ $\eta_k\in L^2(0,T;H^1 (\Omega)),$  $\eta_k=g_k$ on $\Gamma_{D_k},$     	
 	\begin{equation} \label{Stateequations}
 	    \eta_1\vee\eta_2 =:\tilde\eta \in \partial \beta(x,   \rho_1+ \rho_2),\: \quad  \eta_k -\tilde \eta    \in  \partial  I\!\! I_{[0,\infty)}(\rho_k),   \quad \hbox{ a.e. in }Q, 
 	\end{equation} 
 	and    
 	\begin{equation}\label{weakform}
 		\frac{d}{dt}\int_\Omega  \rho _k\: \xi _k    + \int_\Omega (\sigma_k\: \nabla\eta_k+\rho\: V_k+\overline f_k) \cdot \nabla \xi_k  =  \int_\Omega f_{0k}\: \xi_k   +\langle  \pi_k, \xi_k\rangle_{\Gamma_{N_k}}  ,\quad     \hbox{ in }\D('0,T) 
 	\end{equation}  
 	for any $  (\xi_1,\xi_2) \in H^1_{\Gamma_{D_1} } (\Omega) \times  H^1_{\Gamma_{D_2} } (\Omega).$   
 \end{theorem}

 \bigskip 
 
 While a detailed improvement  will be provided in Section~\ref{Smodeling}, we observe  here that the connection between $\rho$ and $\eta $ in   \eqref{Stateequations} hinges essentially on the graph $\partial \B(x,.),$   where for a.e. $x\in \Omega,$ $\B(x,.) \: :\:  \RR^2 \to \RR^+$ is  the application given by 
 $$ \B(x,r)=\left\{ \begin{array}{ll} 
 	\beta(x,r_1+r_2)\quad &\hbox{ if } r_1,\: r_2\geq 0 \\  \\ 
 	+\infty &\hbox{ otherwise} .\end{array} \right. $$  
Remember that $d\in \partial \B(x,r)$   if and only if  $d,\: r\in \RR^2$ and $d\cdot (s-r)\leq \B(x,s)-  \B(x,r),$ for any $s\in \RR^2.$   Working with $ \max_{m\in \RR^2} (m\cdot  q-\beta(x,m)) =:\beta^*(x,q),$   the so called Legendre transform of $\beta(x,.),$   we have 
  
  \begin{proposition}
For a.e. $x\in \Omega,$ $\partial \B(x,.)$ defines a maximal montone graph in $\RR^2\times \RR^2.$ Moreover, for any  $d=(d_1,d_2)$ and $r=(r_1,r_2)\in \RR^2$    we have 
  		\begin{eqnarray*} \label{equivBeta}
  		d\in \partial \B(x,r)\quad
  		&\Leftrightarrow & 	(r_1,r_2)\in \argmax_{s_1,s_2\geq 0 } \Big\{  s_1d_1 + s_2d_2-\beta(x,s_1  +s_2)\Big \}   \\  \\
  		  		&\Leftrightarrow & 		r_1,r_2\geq 0,\: 	  r_1d_1 + r_2d_2-\beta(x,r_1  +r_2)=  \beta^*\left(x,d_1\vee d_2\right)  \\  \\    &\Leftrightarrow&  \quad 		\left\{  \begin{array}{ll} 
  		  			r_1,r_2\geq 0,\: 	d_1\vee d_2 \in \partial\beta \left(x, r_1+r_2\right)   \\ \\  
  		  			r_1\left(d_2-d_1\right)^+ = r_2\left(d_1-d_2\right)^+=0 
  		  			\end{array}\right.  
  		  			 \\  \\ 
  		  			&\Leftrightarrow &    d_1\vee d_2  \in \partial \beta(x,r_1  +r_2)+ ,\: d_k-d_1\vee d_2  \in\partial I\!\! I_{[0,\infty)}(\rho_k)  ,\quad \hbox{ for each }k=1,2.
  	\end{eqnarray*}   
  \end{proposition}
  \begin{proof}  	It is clear that, for a.e. $x\in \Omega,$  	   $\B(x,.)$ is   convex and l.s.c, so that   $\partial \B(x,.)$ is a maximal monotone graph in $\RR^2\times \RR^2.$ Remember that,   $d\in \partial \B(x,r)$ is equivalent to say $r_1,\: r_2\: \geq 0$ and 
  	\begin{eqnarray*}
  		r_1d_1 + r_2d_2-	\underbrace{ \beta(x,r_1+r_2)}_{\B(x,r)} = \underbrace{ \max_{s_1,s_2\geq 0 } \Big\{  s_1d_1 + s_2d_2  -\beta(s_1+s_2  )\Big\} }_{=:\B^*(x,d)} ,
  	\end{eqnarray*}
  	where $\B^*(x,.)$ denotes the Fenchel conjugate of $\B(x,.).$  This implies the first two equivalence assertions. The remaining parts follows readily from the combination of two facts:  the inequality   
  	$$s_1d_1 + s_2d_2 \leq  (s_1+s_2)\:  d_1\vee d_2 \leq  \beta(x,s_1+s_2) +\beta^*\left(x,  d_1\vee d_2\right), $$
  holds to be 	true for any $d_1,\: d_2\in \RR$ and $s_1,s_2\in \RR^+,$ and    the equality   ;  i.e. $r_1d_1 + r_2d_2 =  \beta(x,s_1+s_2) +\beta^*\left(x,  d_1\vee d_2\right),$   holds  to be true  for $r_1,r_2\geq 0$ if and only if    
  	\begin{equation}
  	r_1,r_2\geq 0,\:  	r_1\left( d_2-d_1\vee d_2\right)^+= 	r_2\left( d_1-d_1\vee d_2\right)^+ =  0  \hbox{ and } d_1\vee d_2   \in  \partial\beta \left(x, r_1+r_2\right)   .  
  	\end{equation}

  \end{proof}

 \begin{remark} \label{Rmqequivalentform}
 	To illustrate the relationships between the densities $\rho_1,\: \rho_2,$  and the potentials $\eta_1,\: \eta_2,$  let us explore potential  behaviors of the graph $\partial \B(x,.)$ in $\RR^2 \times \RR^2.$  We present some explicit formulations  of $\partial \B$ and its inverse $(\partial \B(x,.))^{-1},$ which coincides with  $\partial \B^*(x,.). $    
 	\begin{enumerate} 
 	\item 	 For a.e. $x\in \Omega,$ $\partial \B(.,r)$    
 			$$\partial \B(x,r) =  \left\{\begin{array}{ll} 
 				\Big	\{ (d_1,d_2) \: :\:    d_1\vee d_2 \in  \partial \beta(x,0) \Big \}   \quad  & \hbox{ if } r_1=r_2=0\\   
 				\Big	\{(d_1,t)\: :\:  d_1\in \partial \beta_1(x,r_1), \:  t\leq d_1   \Big\} &  \hbox{ if } r_1>0 \hbox{ and } r_2=0\\   
 				\Big	\{( t ,d_2)\: :\:  d_2\in \partial \beta(x,r_2) , \:  t\leq d_2    \Big  \} &  \hbox{ if } r_1=0 \hbox{ and } r_2>0\\  
 				\Big	\{ (t,t) \: :\:   t\in \partial \beta(x,r_1+r_2) \Big \}  &  \hbox{ if } r_1>0 \hbox{ and } r_2>0 ,
 			\end{array}\right. \quad \forall r\in \RR^+\times \RR^+$$ 
  and 
 			$$\partial \B^*(x,d) =  \left\{\begin{array}{ll} 
 				\Big	\{ (0,t) \: :\:   t\in     \partial \beta^*(x,d_2) \Big \}   \quad  & \hbox{ if } d_1<d_2\\   
 				\Big	\{ (t,0) \: :\:   t\in     \partial \beta^*(x,d_1) \Big \}   \quad  & \hbox{ if } d_1>d_2\\  
 				\Big	\{ (t_1,t_2) \: :\:  t_1\geq 0,\: t_2\geq 0,\: t_1+t_2\in     \partial \beta^*(d_1) \Big \}   \quad  & \hbox{ if } d_1=d_2\\  
 			\end{array}\right. \quad \forall d\in \RR\times \RR.$$

 			\item In the case where $\beta(x,.)$  (resp. $\beta_1^*(x,.)$) is differentiable, 	 for a.e. $x\in \Omega,$ we have 
 			$$\partial \B (x,r) =  \left\{\begin{array}{ll} 
 				\Big	\{ (d_1,d_2) \: :\:    d_1\vee d_2 =0 \Big \}   \quad  & \hbox{ if } r_1=r_2=0\\   
 				\Big	\{(\beta'(x,r_1),t)\: :\:     t\leq \beta'(x,r_1)    \Big\} &  \hbox{ if } r_1>0 \hbox{ and } r_2=0\\  
 				\Big	\{( t ,\beta'(x,r_2))\: :\:    t\leq \beta'(x,r_2)  \Big  \} &  \hbox{ if } r_1=0 \hbox{ and } r_2>0\\   
 				\Big (\beta'(r_1+r_2),\beta'(x,r_1+r_2)\Big)    &  \hbox{ if } r_1>0 \hbox{ and } r_2>0 ,
 			\end{array}\right.\quad \forall r\in \RR^+\times \RR^+ $$   		
 		and 
 			$$\partial \B^*(x,d) =  \left\{\begin{array}{ll} 
 				(0,(\beta^*)'(x,d_2)   \quad  & \hbox{ if } d_1<d_2\\   
 				((\beta^*)'(x,d_1),0)   \quad  & \hbox{ if } d_1>d_2\\  
 				\Big	\{ (t_1,t_2) \: :\:  t_1\geq 0,\: t_2\geq 0,\: t_1+t_2=        (\beta^*)'(x,d_1)  \quad  & \hbox{ if } d_1=d_2\\   
 			\end{array}\right.\quad \forall d\in \RR \times \RR  .$$

 				\item For instance in the case where $\beta(r)=\frac{1}{m+1} r^{m+1}$ (porous medium equation)	 $\partial \B(r) $    may be defined by  
 			$$\partial \B(r) =  \left\{\begin{array}{ll} 
 				\Big	\{ (d_1,d_2) \: :\:    d_1\vee d_2 =0 \Big \}   \quad  & \hbox{ if } r_1=r_2=0\\   
 				\Big	\{(r_1^m,t)\: :\:     t\leq r_1^m   \Big\} &  \hbox{ if } r_1>0 \hbox{ and } r_2=0\\  
 				\Big	\{( t ,r_2^m)\: :\:    t\leq r_2^m \Big  \} &  \hbox{ if } r_1=0 \hbox{ and } r_2>0\\   
 				\Big ( (r_1+r_2)^m, (r_1+r_2)^m\Big)    &  \hbox{ if } r_1>0 \hbox{ and } r_2>0 ,
 			\end{array}\right.\quad \forall r\in \RR^+\times \RR^+ .$$

 				\item In the case where $\beta=I\!\!I_{[0,1]}$, which corresponds to crowed motion models, 	 $\partial \B(r) $    may be defined by  
 			$$\partial \B(r) =  \left\{\begin{array}{ll} 
 				\Big	\{ (d_1,d_2) \: :\:    d_1\vee d_2 =0 \Big \}   \quad  & \hbox{ if } r_1=r_2=0\\   
 				\Big	\{(\eta_1,t)\: :\:     t\leq \eta_1,\: \eta_1\in \partial I\!\!I_{[0,1]} (r_1)  \Big\} &  \hbox{ if } r_1>0 \hbox{ and } r_2=0\\  
 				\Big	\{( t ,\eta_2)\: :\:    t\leq \eta_2,\: \eta_2\in \partial I\!\!I_{[0,1]} (r_2)  \Big  \} &  \hbox{ if } r_1=0 \hbox{ and } r_2>0\\   
 			\Big	\{( t ,t)\: :\:    \: t\in \partial I\!\!I_{[0,1]} (r_1+r_2)   \Big  \}    &  \hbox{ if } r_1>0 \hbox{ and } r_2>0 ,
 			\end{array}\right.\quad \forall r\in \RR^+\times \RR^+ .$$

%
 				%
 				%
 			%
 			%
 			%
 			%
 			%
 			%
 			%
 			%
 			%
 			
 		\end{enumerate}

 \end{remark}

 \begin{remark}\label{RspatialR}
Assume for instance that  $\overline f$ is equal to $0,$ and $\beta$ is differentiable.  On sees  that the dynamic of the system \eqref{PDEevol}   splits formally into three driver regions, $S_1:=[\rho_1>0\: \: \&\:\:  \rho_2=0]$, $S_2:=[\rho_1=0\: \: \&\: \: \rho_2>0]$ and $S:=[\rho_1>0\: \: \&\: \: \rho_2>0]$ each exhibiting specific behavior.  Assuming they are sufficiently regular, we can formally observe the following behavior  

  	\begin{minipage}[t]{0.49\textwidth} \vspace{0pt}
 		\begin{mytbox}{$${S_1=[\rho_1>0\: \&\: \rho_2=0]}$$}
 			$\left\{ \begin{array}{lll} 
 				\partial_t \rho_1- \nabla\cdot ( \sigma_1\: \nabla \beta'(\rho_1)+ \rho_1V_1)=f_1 \\ \\ 
 				- \nabla \cdot(  \sigma_2\:\nabla  \eta_2)=f_2,\:   \eta_2 \leq   \beta'(\rho_1) \end{array} \right. $
 		\end{mytbox}
 	\end{minipage} 
 	\begin{minipage}[t]{0.49\textwidth} \vspace{0pt}
 		\begin{mytbox}{$${S_2= [ \rho_1=0\: \&\: \rho_2>0 ]}$$}
 			$\left\{\begin{array}{lll} 
 				-\nabla \cdot( \sigma_1\:  \nabla \eta_1)=f_1,\:  \eta_1 \leq  \beta'(\rho_2)\\  \\ 
 				\partial_t  	\rho_2- \nabla\cdot (   \sigma_2\: \nabla \beta'(  \rho_2)+ \rho_2V_2) =f_2 	\end{array}  \right. $
 		\end{mytbox}
 	\end{minipage}  
 	
 	\centerline{	\begin{minipage}[t]{0.6\textwidth} \vspace{0pt}
 			\begin{mytbox}{$${S= [\rho_1>0\:  \& \:  \rho_2>0]}$$}
 				$\left\{\begin{array}{lll} 
 					\partial_t 	\rho_1- \nabla\cdot ( \sigma_1\: \nabla \beta'(\rho_1+\rho_2)+ \rho_1V_1)  =f_1	\\  \\  
 					\partial_t 	\rho_2-   \nabla\cdot (   \sigma_2\: \nabla \beta'(\rho_1+\rho_2)+ \rho_2V_2) =f_2	\end{array} \right.$
 			\end{mytbox}
 	\end{minipage} }
 
 	The system in equation \eqref{PDEevol} exhibits similarities to the commonly used system in the literature shown in 
 	\begin{equation} \label{PDEevol0}
 		\left\{ \begin{array}{ll}
 			\partial_t \rho_1-   \nabla\cdot(   \sigma_1\:   \rho_1  \:  \nabla \varphi' (\rho_1+\rho_2)+\rho_1  \:V_1) = f_1\\   \\ 
 			\partial_t \rho_2-  \nabla\cdot(     \sigma_2\:   \rho_2   \: \nabla \varphi' (\rho_1+\rho_2)+\rho_1  \: V_1) = f_2 .
 		\end{array}\right.  
 	\end{equation} 
 	Both systems share some  similar behaviors in the segregated regions, $S_1,$ and $S_2$ as long as the function $\varphi$  satisfies $\beta(r)=r\: \varphi(r)-r,$ for any $r\in \RR.$  However, there are key differences. For instance :\\
 	- \underline {Compatibility with source/reaction terms:} Equation \eqref{PDEevol} appears to be more compatible and adaptable to include more possibilities for source/reaction terms compared to \eqref{PDEevol0}, even in segregated regions. \\
 	- \underline{Parabolic nature:} In the aggregated region, $S,$ the systems diverge significantly. While  \eqref{PDEevol} exhibits a fully parabolic regime with respect to all three variables $\rho_1,$ $\rho_2$ and $\rho_1+\rho_2,$   equation \eqref{PDEevol0} is only parabolic with respect $\rho_1+\rho_2.$  \\
 	 	- \underline{Correlation of solutions in the segregation regime:}  We observe that if a solution $(\rho,\eta)$  to  \eqref{PDEevol} satisfies the condition $[\rho_1\rho_2\neq 0]=\emptyset,$ (i.e., purely segregation regime), then it is also a weak solution to  \eqref{PDEevol0}. However, it is not yet clear if the converse implication holds true in general. 	The existence of a segregation regime for the system described by Equation~\eqref{PDEevol} is an intriguing question that we do not address in this paper.

%
%
 	 
 \end{remark}

  \begin{remark}
  	Since we're only interested in nonnegative solutions, the behavior of the function $\beta(x,r)$ for negative values of $r$    is irrelevant. We can assume that $\beta(x,r)=0$ for any $r\leq 0.$  As a consequence, for any $s\in \partial \beta(x,r),$ we have  $s\geq 0.$   This nonnegativity property has a crucial implication for solutions of equation (\ref{PDEevol}). Specifically, for any solution $(\rho,\eta)$   we can systematically conclude that:
  	\begin{equation}
  		\label{signeta}
  		\eta_1 \vee \eta_2 \geq 0 \quad \text{ a.e.  in } Q.
  	\end{equation}

   \end{remark}
  
%
%
%
%
%
%
 
 \subsection{$H^{-1}-$gradient flow approach}\label{sectionmainrH-1}
   
 An interesting future  direction of this approach lies in establishing a formal connection between the dynamics in equation \eqref{PDEevol} and $H^{-1}-$gradient flow, mirroring the relationship observed in single-species models. We demonstrate this connection in the context of two species, thereby providing a general abstract framework for cross-diffusion systems arising from overcrowding dispersal phenomena.   For each $k=1,2,$ we consider the dual Sobolev space $H^{-1}_{\Gamma_{D_k}}(\Omega)$. It is not difficult to see that 
 \begin{equation} 
	\Vert f_k\Vert _{H^{-1}_{\Gamma_{D_k}}} =  	\min_{\omega\in L^2(\Omega)^N}\left \{  \left(\int_\Omega \sigma_k\: \vert   \omega   \vert^2\right)^{1/2}    \:  :\:  - \nabla \cdot ( \sigma_k\: \omega+ \overline f_k )   =f_{0k}  \hbox{ in }\Omega \hbox{ and }   ( \sigma_k\: \omega  +\overline f_k )\cdot \nu =0 \hbox{ on }\Gamma_{N_k }  \right\}  ,
\end{equation}  for any  $f\in [H^{-1}_{\Gamma_{D}}(\Omega)]^2,$ defines a norm on  $H^{-1}_{\Gamma_{D_k}}(\Omega).$  This norm is clearly associated with the inner product   
$$\langle f_k,g_k\rangle_{k,\Omega} : =  \int_\Omega\sigma_k \: \phi^f_k\cdot \phi^g_k , \quad \hbox{ for any }f,g\: \in H^{-1}_{\Gamma_{D_k}}(\Omega)$$
where $\phi^f_k$ and $\phi^g_k$ are (unique) given by 
$$\phi^f_k =  \argmin_{\omega \in L^2(\Omega)^N}\left \{   \int_\Omega    \sigma_k\:\vert  \omega   \vert^2    \:  :\:  - \nabla \cdot (\sigma_k\:\omega+ \overline f_k )   =f_{0k}  \hbox{ in }\Omega \hbox{ and }   (\sigma_k\: \omega  +\overline f_k )\cdot \nu =0 \hbox{ on }\Gamma_{N_k }  \right\}   $$
and 
$$\phi^g_k =  \argmin_{\omega \in L^2(\Omega)^N}\left \{   \int_\Omega\sigma_k\: \vert   \omega   \vert^2    \:  :\:  - \nabla \cdot (\sigma_k\: \omega+ \overline g_k )   =g_{0k}  \hbox{ in }\Omega \hbox{ and }   (\sigma_k\:\omega  +\overline g_k )\cdot \nu =0 \hbox{ on }\Gamma_{N_k }  \right\} .  $$  
Remember here (see for instance Proposition \ref{Lduals}), that   there exists a unique $\eta^f_k\in H^1_{\Gamma_{D_k}}\! \!(\Omega)$  (resp. $\eta^g_k\in H^1_{\Gamma_{D_k}}\! \!(\Omega)$) such that $\phi^f_k= \nabla \eta^f_k$ ( resp. $\phi^g_k= \nabla \eta^g_k$) a.e. in $\Omega.$ 
 This implies that  $H^{-1}_{\Gamma_{D_k}}\! \!(\Omega)$ equipped with  the norm $\Vert .\Vert _{H^{-1}_{\Gamma_{D_k}}}  $ and the inner product  $\langle .,.\rangle_{k,\Omega}$ is     a Hilbert space.   
   
  This being said,  let us consider the space  
  $H^{-1}_{\Gamma_{D_1}}\! \!(\Omega)\times H^{-1}_{\Gamma_{D_2}}\! \!(\Omega) ,$ endowed with the norm 
 $$\Vert f\Vert_{H^{-1}_{\Gamma_{D}}}  =\left( \Vert f_1\Vert _{H^{-1}_{\Gamma_{D_1}}}^2  +\Vert f_2\Vert _{H^{-1}_{\Gamma_{D_2}}} ^2   \right)^{1/2}, \quad \hbox{ for any }f=(f_1,f_2)\in  H^{-1}_{\Gamma_{D_1}}\! \!(\Omega)\times H^{-1}_{\Gamma_{D_2}}\! \!(\Omega) .$$
This is in turn  a Hilbert space, and the    associate  inner product  is given by
$$[ f,g]_{\Omega} =  \langle f_1,g_1\rangle_{1,\Omega}  + \langle f_2,g_2\rangle_{k,\Omega} ,\quad \hbox{ for any }f,g\in   H^{-1}_{\Gamma_{D_1}}\! \!(\Omega)\times H^{-1}_{\Gamma_{D_2}}\! \!(\Omega).$$ 

\bigskip 

Now, let us define the functional $\E\: :\: H^{-1}_{\Gamma_{D_1}}\! \!(\Omega)\times H^{-1}_{\Gamma_{D_2}}\! \!(\Omega)  \to [0,\infty]$  by 
\begin{equation}\label{enrgy}
	 \E(\rho):= \left\{ \begin{array}{ll}
	\int_\Omega \beta(.,\rho_1+\rho_2)\quad  &\hbox{ if }\rho_1,\: \rho_2,\:  \beta(.,\rho_1+\rho_2) \in L^1_+(\Omega) \\  \\
	+\infty  &\hbox{ otherwise }.	\end{array}   \right. 
\end{equation}  
Thanks to   assumptions $(H1)$ and $(H2)$, it is not difficult to  see that $\E$ is convex and  l.s.c. in   $H^{-1}_{\Gamma_{D_1}}\! \!(\Omega)\times H^{-1}_{\Gamma_{D_2}}\! \!(\Omega) .$ This   implies that the sub-differential   of $\E,$ $\partial\E,$ is well-defined  and forms a maximal monotone graph within   $ \left(H^{-1}_{\Gamma_{D_1}}\! \!(\Omega)\times H^{-1}_{\Gamma_{D_2}}\! \!(\Omega), \Vert -\Vert_{H^{-1}_{\Gamma_{D}}} \right). $  So, we can consider the $H^{-1}_{\Gamma_{D_1}}\! \!(\Omega)\times H^{-1}_{\Gamma_{D_2}}\! \!(\Omega)-$gradient flow  problem 
\begin{equation}\label{EvolH-1}
	\left\{  \begin{array}{ll}
		\frac{d}{dt}\rho (t)+ \partial \E(\rho(t))\ni h(t) \quad  & \hbox{ a.e.  }t\in (0,T) \\  \\ 
		\rho(0)=\rho_0,
	\end{array}\right. 
\end{equation} 
where $h\in L^2(0,T;H^{-1}_{\Gamma_{D_1}}(\Omega)\times H^{-1}_{\Gamma_{D_2}}(\Omega) ).$
Our aim now, is to establish the connection  between  the problems \eqref{PDEevol} and  \eqref{EvolH-1}.   This is summarized in the following theorem.

  \begin{theorem}\label{ThH-1} Under the assumptions   $g\equiv 0$,   $\pi\equiv 0$,   $f\in L^2\left(0,T; H^{-1}_{\Gamma_{D_1}}\! \!(\Omega)\times H^{-1}_{\Gamma_{D_2}}\! \!(\Omega)\right)$ and $V\in [L^\infty(Q)]^2,$      let   us consider $(\rho,\eta)$ be a weak solution of  the system of PDE  \eqref{PDEevol} given by Theorem \ref{texistuniq}. Then, $\rho$    coincides with the   strong solution of  \eqref{EvolH-1}, where 
  	$$h= (f_1-\nabla \cdot (\rho_1\: V_1) , f_2-\nabla \cdot ( \rho_2\: V_2 )).$$  
  \end{theorem}

Thanks to general theory of evolution problem governed by maximal monotone graph on Hilbert space,  this theorem implies in particular that, if $(\rho^1,\eta^1)$ and $(\rho^2,\eta^2)$ are  weak solutions of  the system of PDE  \eqref{PDEevol} corresponding to given $(f^1,V^1)$ and $(f^2,V^2),$ respectively, then   
 \begin{equation}\label{Hcomparison}
 	\begin{array}{c}  	\frac{1}{2}	\frac{d}{dt} \sum_{k=1,2} \Vert \rho^1_k-\rho^2_k \Vert_{H^{-1}_{\Gamma_{D_k}}(\Omega) }^2 + 
 		[\partial \E(\rho^1) - \partial \E(\rho^2),  \rho^1-\rho^2]_\Omega  \\  
 		\leq    \sum_{k=1,2}  \langle f^1_k-f^2_k,  \rho^1_k-\rho^2_k  \rangle_{ k,\Omega} -   \sum_{k=1,2}  \langle \nabla \cdot(\rho^1_k\: V^1_k-\rho^2\: V^2_k),  \rho^1_k-\rho^2_k  \rangle_{ k,\Omega}   ,\quad \hbox{ in }\D'(0,T).
 \end{array} 	\end{equation}
 More specifically, this implies the following  uniqueness result. 
 
 \begin{corollary}\label{Corcontraction}
Under the assumptions   $g\equiv 0$   $\pi\equiv 0$ and $ f  \in L^2\left(0,T; H^{-1}_{\Gamma_{D_1}}\! \!(\Omega) \times H^{-1}_{\Gamma_{D_2}}\! \!(\Omega)\right)$, let us consider  $V^1,\: V^2\in L^\infty(Q)$.    If  $(\rho^1,\eta^1)$ and $(\rho^2,\eta^2)$ are  weak solutions of  the system of PDE  \eqref{PDEevol} corresponding to $ V^1 $ and $ V^2 ,$ respectively, such that $\rho^1(0)=\rho^2(0)$ and there exists $c>0$ such that 	\begin{equation}\label{suffuniq}
	\begin{array}{c}   		[\partial \E(\rho^1) - \partial \E(\rho^2),  \rho^1-\rho^2]_\Omega +   \sum_{k=1,2}  \langle \nabla \cdot(\rho^1_k\: V^1_k-\rho^2\: V^2_k),  \rho^1_k-\rho^2_k  \rangle_{ k,\Omega}  \\  \geq -  c \: \sum_{k=1,2}\Vert \rho^1_k-\rho^2_k \Vert_{H^{-1}_{\Gamma_{D_k}}(\Omega) }^2,\quad \hbox{ a.e. in  }(0,T),
\end{array}  \end{equation}
then  $\rho_1=\rho_2.$  
 \end{corollary}
 \begin{proof}The proof is a simple consequence of   \eqref{Hcomparison} and Grönwall inequality.  	
 	\end{proof}

  \bigskip

 Equation~\eqref{suffuniq} incorporates assumptions about $\beta$ and $V^s,$ for $s=1,2.$   This condition becomes  particularly relevant in applications where the drift term depends on the densities themselves. However, in the special case where $V\equiv 0$ (i.e., no potential term), the uniqueness of the weak solution can be established without additional assumptions. This directly follows from the monotone property of $\partial \E$, expressed as :
  $$[\partial \E(\rho^1(t)) - \partial \E(\rho^2(t)), \rho^1 - \rho^2]_\Omega \geq 0.$$
  Therefore, the following uniqueness result holds.
  
 \begin{corollary}\label{Coruniquness}
 	Under the assumptions $g\equiv 0,$   $\pi\equiv 0$, $V\equiv  0,$    $f\in L^2\left(0,T; H^{-1}_{\Gamma_{D_1}}\! \!(\Omega)\times H^{-1}_{\Gamma_{D_2}}\! \!(\Omega)\right)$ and  $ \rho_0 \in L^{2}(Q)\times L^{2}(Q) $ satisfying 
\eqref{Condrho0},     the system of PDE  \eqref{PDEevol} admits a unique weak solution. 
 \end{corollary}

 \begin{remark} \label{remarqyeBeta}
 \begin{enumerate}
 	\item Theorem \ref{ThH-1} bridges a connection  between the cross-diffusion system \eqref{PDE1} and the well-established theory of gradient flows in Hilbert spaces.  Furthermore, the specific form of the internal energy \eqref{enrgy} allows us to definitely connect the differential operator $\A$ (defined precisely in Proposition \ref{PH-1}) governing \eqref{PDE1} to the dispersal dynamics of overcrowding for two species.  As to the connection with systems of type \eqref{crossgen2} and \eqref{PDEevol0} commonly used in the literature, we de believe that  is likely more nuanced and may hold true only in specific cases, particularly when solutions exist, such as in fully segregated scenarios.

 	 \item  Theorem \ref{ThH-1} also yields other crucial results which uses the equivalence between weak solutions and $H^{-1}_{\Gamma_{D_1}}\! \!(\Omega)\times H^{-1}_{\Gamma_{D_2}}\! \!(\Omega)-$gradient flow solutions. Beyond uniqueness results of Corollary \ref{Corcontraction} and Corollary \ref{Coruniquness}, the large time behavior which could be of great importance for the applications maybe concerned too.  Indeed, thanks to the theory of gradient flows in Hilbert spaces, we see  that if $f\in L^2(0,\infty ;H^{-1}_{\Gamma_{D_1}}\! \!(\Omega)\times H^{-1}_{\Gamma_{D_2}}\! \!(\Omega)) $ is such that, as $t\to\infty,$   
 	$$f(t)\to f_\infty,\quad \hbox{ in } H^{-1}_{\Gamma_{D_1}}\! \!(\Omega)\times H^{-1}_{\Gamma_{D_2}},$$
 	then the solution $\rho(t)$ converges uniformly in $H^{-1}_{\Gamma_{D_1}}\! \!(\Omega)\times H^{-1}_{\Gamma_{D_2}}\! \!(\Omega)$  to $\rho,$ as $t\to\infty,$ where $ \rho \in L^{2}(\Omega)\times L^{2}(\Omega) $ satisfies  
 	\eqref{Condrho0} and there exists $\eta\in [H^1(\Omega)]^2$ for which the couple $(\rho,\eta)$ is a solution of the stationary problem 
 	 \begin{equation} 
 		\left\{\begin{array}{ll} 
 			\left.  \begin{array}{ll} 
 	 - \nabla\cdot ( \sigma_k\:  \nabla\eta _k+  \overline f_{\infty k})=  f _{\infty 0k}  ,  \quad  \hbox{ for }k=1,2 \\ 
 				\\  		\eta_1\vee\eta_2 =:\tilde\eta \in \partial \beta(x,   \rho_1+ \rho_2),\: \quad  \eta_k -\tilde \eta    \in  \partial  I\!\! I_{[0,\infty)}(\rho_k),   
 			\end{array}\right\} 	\quad & \hbox{ in } Q, \\  \\
 			(\sigma_k\: \nabla\eta_k + \overline f_{\infty k})\cdot \nu  =  \pi_k   & \hbox{ on }\Sigma_{N _k}\\  \\
 			\eta_k  =  g _k ,\quad    & \hbox{ on }\Sigma_{D _k}. 
 	\end{array}	\right. 	 
 	\end{equation}

 	\item  In the case where $V\not\equiv 0,$ the application of $ H^{-1}_{\Gamma_{D_1}}\! \!(\Omega)\times H^{-1}_{\Gamma_{D_2}}\! \!(\Omega)-$gradient flow theory  remains to be restrictive. The uniqueness of weak solutions of the problem \eqref{PDE1} remains to be a challenging open problem. Remember that for the case of one specie, the uniqueness may follows  through contraction principles in $L^1$ and/or in  $\mathcal W_2-$Wasserstein space (one can see \cite{IgShaw,KM} for discussions and references in this direction).    It is not clear yet for us how to deal with $L^1-$theory for cross-diffusion of the type \eqref{PDE1}.

 	\item In this paper, we focus on the case where source term  $f=f(t,x).$ However, one sees that one can use the results of Theorem \ref{ThH-1} also  for treating more general reaction terms $f=f(t,x,\rho,\eta).$  
 	
 	\item At first glance,  at least for overcrowding  dynamics of one specie, the theories of gradient flows in $H^{-1}$ and $\mathcal W_2-$Wasserstein space   appear to be similar, providing comparable dynamic behaviors through more or less similar internal energies.  It is surprising to note here that, this seems not to be the case for two species. Indeed,  recent works
 	 suggests that the theory of gradient flows in  $\mathcal W_2-$Wasserstein space appears to be closer to the model  
   \eqref{PDEevol0}  à la  Gurtin et al. (cf. \cite{Gurtin84})   	than the $H^{-1}$-theory, which provides complete and general results for \eqref{PDE1}  and seems to generalize  \eqref{PDEevol0}.  We believe that further developments using $\mathcal W_2$ and $H^{-1}$ frameworks should be carried out to complete or explain more rigorously the links between both approaches in the case of two species. 
 	
 \end{enumerate}
 \end{remark}

  \subsection{Weighted internal energy}
  The results of Section \ref{sectionmainresultexist} and Section \ref{sectionmainrH-1}  enable  also to treat the case where the densities are combined with specific weights in the internal energy. These weights represent the relative importance of each density in influencing the internal energy.  In this case the cross-diffusion system reads 
  \begin{equation}
  	\label{PDEevolalpha}
  	\left\{\begin{array}{ll} 
  		\left.  \begin{array}{ll} 
  			\partial_t 	\rho _k- \nabla\cdot ( \sigma_k\:   \nabla\eta _k+ \rho_k V_k+\overline f_k)=  f _{0k}  ,  \quad  \\ 
  			\\    \eta_k    \in \partial \beta(x,  \alpha_1\rho_1+\alpha_2\rho_2)   + \partial  I\!\! I_{[0,\infty)}(\rho_k)
  		\end{array}\right\} 	\quad & \hbox{ in } Q, \\  \\
  		(\sigma_k\: \nabla\eta_k + \rho_k V_k +\overline f_k)\cdot \nu  =  \pi_k   & \hbox{ on }\Sigma_{N _k}\\  \\
  		\eta_k  =  g _k ,\quad    & \hbox{ on }\Sigma_{D _k},\\  \\ 
  		\rho_k(0)=\rho_{0k}  & \hbox{ in } \Omega, 
  	\end{array}
  	\right\}  	   \hbox{ for }k=1,2 ,
  \end{equation}   
  where $(\alpha_1,\alpha_2)\in \RR^2$ are given and assumed to satisfy the following condition 
  \begin{equation} 
  	0< \min(\alpha_1 ,\alpha_2 ) .  
  \end{equation} 
  
  \begin{theorem} \label{Cexistuniq} Under the previous assumptions, assume moreover that 
  	$$\int_\Omega \beta(.,\alpha_1\rho_{01}+\alpha_2\rho_{02})<\infty.$$
  	Then,  the system of PDE  \eqref{PDEevolalpha}    has a solution $(\rho,\eta)$ in  the sense that,        for each $k=1,2,$  $\rho_k \in L^2_+(Q) \cap  W^{1,2}(0,T;H^{-1}_{\Gamma_{D_k}}(\Omega)),$  $\rho_k(0)=\rho_{0k},$  $\eta_k\in L^2(0,T;H^1 (\Omega)),$  $\eta_k=g_k$ on $\Gamma_{D_k},$     	
  	\begin{equation} \label{Stateequations1}
  		\left\{ \begin{array}{l}    \eta_1  \vee  \eta_2  =:\tilde\eta \in \partial \beta(x, \alpha_1 \rho_1+\alpha_2\rho_2),\: \\  \\    \eta_k -\tilde \eta    \in  \partial  I\!\! I_{[0,\infty)}(\rho_k),
  		\end{array} \right.      \quad \hbox{ a.e. in }Q, 
  	\end{equation} 
  	and    
  	\begin{equation}\label{weakform1}
  		\frac{d}{dt}\int_\Omega  \rho _k\: \xi _k    + \int_\Omega (\sigma_k\: \nabla\eta_k+\rho\: V_k+\overline f_k) \cdot \nabla \xi_k  =  \int_\Omega f_{0k}\: \xi_k   +\langle  \pi_k, \xi_k\rangle_{\Gamma_{N_k}}  ,\quad     \hbox{ in }\D('0,T), 
  	\end{equation}  
  	for any $  (\xi_1,\xi_2) \in H^1_{\Gamma_{D_1} } (\Omega) \times  H^1_{\Gamma_{D_2} } (\Omega).$ 
  \end{theorem}
 \begin{proof}  The proof of this result follows directly by   using Theorem \ref{texistuniq} and the fact that $(\rho,\eta)$ is a solution of \eqref{PDEevolalpha} if and only if the couple $(\tilde \rho,\eta),$ where $\tilde \rho:=(\alpha_1\:\rho_1,\alpha_2\:\rho_2),$ is a solution of the  
 	 \begin{equation}
 		\left\{\begin{array}{ll} 
 			\left.  \begin{array}{ll} 
 				\partial_t  \tilde 	\rho _k- \nabla\cdot (  \sigma_k/\alpha_k\:   \nabla\eta _k+ \tilde  \rho_k  V_k+ \overline f_k/\alpha_k)=  f _{0k}/\alpha_k  ,  \quad  \\ 
 				\\    \eta_k    \in \partial \beta(x, \tilde\rho_1+\tilde\rho_2)   + \partial  I\!\! I_{[0,\infty)}(\tilde \rho_k)
 			\end{array}\right\} 	\quad & \hbox{ in } Q, \\  \\
 			(\sigma_k/\alpha_k\: \nabla\eta_k + \tilde \rho_k V_k +\overline f_k/\alpha_k)\cdot \nu  =  \pi_k/\alpha_k   & \hbox{ on }\Sigma_{N _k}\\  \\
 			\eta_k  =  g _k ,\quad    & \hbox{ on }\Sigma_{D _k},\\  \\ 
 			\rho_k(0)=\rho_{0k}  & \hbox{ in } \Omega, 
 		\end{array}
 		\right\}  	   \hbox{ for }k=1,2 .
 	\end{equation}    
 	\end{proof}

   \section{Modeling through proximal minimization algorithm}
   \label{Smodeling}

   \subsection{Preliminaries }
   As introduced earlier, our approach to cross-diffusion modeling relies on a proximal minimization algorithm. This algorithm, commonly used in the context of nonlinear semigroups within Banach/Hilbert spaces, emerges from the Euler-implicit scheme. It also plays a crucial role in the JKO scheme, minimizing movement within the gradient flow in Wasserstein space, and showcases connections to the gradient descent method used for optimization. In a previous work \cite{IgGF}, we prove its application to establish a general nonlinear diffusion equation for a single species. This paper aims to extend this approach to the case of two interacting species. We refer readers to  \cite{IgGF} for further details, discussions, development, and references related to this method.

Recall that      proximal minimization algorithms target minimizing the system's internal energy against the work associated with transitions from a given state to a next one. Following \cite{IgGF}, we propose employing the minimum flow process to measure this work. This process enables to entail the appropriate dynamic procedure that continuously updates the system's state, aiming to decrease its internal energy through a nonlinear diffusion process.

   \medskip

   More formally,   let $X$ be a state space and $\mathcal E(\rho)$ be the internal energy of a given system  at state $\rho\in X.$   At each time $t>0,$  the dynamic system tends simultaneously  to decrease its internal energy and to minimize the work required to move from state $\rho(t)$ to state $\rho(t+h).$  This can be achieved by solving the following optimization problem (proximal energy) at each time step:
   $$\min_{\rho \in X}  \Big\{  \mathcal E(\rho) + W^h(\rho(t), \rho) \Big\} ,$$ 
   where $W^h(\rho(t),\rho)$ is a measure of the work required to move from state $\rho(t)$ to state $\rho,$ throughout small time step $h>0.$  The solution to this optimization problem $\rho(t+h) $ is the next state of the system.   To the point, one can  consider the  $\tau-$time steps defined by  $t_0 = 0 < t_1<\cdots < t_n <T$, and    the sequence of piecewise constant curves
   $$\rho_\tau =\sum_{i=1}^n\rho_i   \:  \chi_{]t_{i-1}  ,t_{i}]}  , $$
   where  \begin{equation}\label{steepestdescent0}
   	\rho_i= \argmin_{\rho \in X} \Big\{  \mathcal E  ( \rho) + W^\tau(\rho_{i-1}, \rho)\Big\} ,\quad \hbox{   for }i=1,...n, \hbox{ with }\rho_0=\rho(0).
   \end{equation}
   Then, we expect  the limit as $\tau\to 0$  of $\rho_\tau$ converges to  the solution of a continuous evolution problem.   
   Working within the framework of   Hilbert  space $(X,\Vert -\Vert),$ the general theory suggests to connect $W^\tau(\rho(t), \rho)$ to $\Vert\rho(t)- \rho \Vert^2/2\tau $   to  obtain a representation of the solution through  the gradient flow equation:
   \begin{equation}\label{evol0}
   	\left\{ \begin{array}{ll} \rho_t(t)    +\partial \mathcal E(\rho(t)) \ni 0\quad \hbox{ in }(0,T) \\  \\ 
   		\rho(0)\hbox{ is given},  \end{array}\right. 
   \end{equation} 
   where $\partial \mathcal E(\rho)$ denote sthe sub-differential (possibly multi-valued map) of the functional $\rho\in X\to \mathcal E(\rho)\in (-\infty,\infty].$   The expression \eqref{steepestdescent0} can be related to   both to the resolvent associated with the operator $\partial \mathcal E $  and Euler-Implicit discretisation of \eqref{evol0}.  The   evolution problem  in turn  is   a gradient flow  in a Hilbert  space  which describes the evolution of a continuous curve $t\to \rho(t)\in X$ that follows the direction of steepest descent of the  functional $\mathcal E$ (cf. \cite{Br}).  It is possible also to consider a more general situation where  $(X,d)$ is metric space and  $W^h$  is build on the distance $d$. In this case, the expression \eqref{steepestdescent0} reveals   the minimizing movement scheme à la  De Giorgi (see \cite{AGS}). 
   The approach resulted in a broad theory of gradient flows in a metric space where the derivatives $  \rho_t  $ and $ \partial  \mathcal E(\rho)$ need to be interpreted in an appropriate   way in $(X,d)$ (cf. the book \cite{AGS}).

   \medskip  
   For instance, the porous medium-like equation
   \begin{equation}\label{PME} 
   	\partial_t \rho -\Delta \rho^m = 0 
   \end{equation}
   in a bounded domain  arises in various applications, especially in biology, to represent the evolution  of species as they strive to minimize their internal energy $\mathcal E(\rho),$ which depends inherently on the density $\rho.$  Since the works \cite{OttoPME} and \cite{JKO}, it has been established that the PDE \eqref{PME} subject to homogeneous Neuman boundary condition can  be derived by employing $\mathcal E(\rho)=\frac{1}{m-1}\int_\Omega  \rho^m $  in the context of $W_2-$Wasserstein distance by setting $W^\tau (\rho,\tilde \rho)= \frac{1}{2\tau} \mathcal W_2( \rho,\tilde \rho)^2$.   
   It is well known that it can be also derived by utilizing the steepest descent algorithm for the internal energy $\mathcal E(\rho)=\frac{1}{m+1}\int_\Omega  \rho^{m+1} $  in the context of $H^{-1}-$norm.  This is achieved by setting $W^\tau (\rho,\tilde \rho)= \frac{1}{2\tau} \Vert \rho-\tilde \rho\Vert _{H^{-1}}^2$  (cf. \cite{Barbu,IgGF}).     These concepts can be expanded to accommodate some monotone nonlinearity $\eta(r)$ instead of $r^m,$ both for $W_2-$Wasserstein distance and also $H^{-1}-$norm   frameworks (cf. \cite{KM,Dam,DK,Kenmochi}).  These frameworks  can be also extended to nonlinear diffusion operators (see \cite{Agueh}  for the $W_c-$Wasserstein distance framework, and the recent work \cite{IgGF} for the the case a generalized dual Sobolev  approach). 
     
   \medskip
    This paper focuses on a system with two species $(k=1,2)$, represented by a couple of densities, $(\rho_1, \rho_2),$ in  the state space   $X=X_1\times X_2.$  To capture the overcrowding dispersal dynamic, we consider an internal energy function dependent on the sum of both densities ; i.e.  $\E(\rho)=\tilde \E(\rho_1+\rho_2).$  
    We employ a proximal optimization algorithm to construct a sequence of density pairs $(\rho_1 ,\rho_2)$ starting from initial densities $(\rho_{01}, \rho_{02})$. The algorithm iteratively updates the densities using the following steepest descent formula:
   	\begin{equation}\label{steepestdescent1}
   	\rho^i= \argmin_{\rho \in X} \Big\{  \tilde \E(\rho_1+\rho_2)   + W^\tau_1(\rho_1^{i-1}, \rho_1)+ W^\tau_2(\rho_2^{i-1}, \rho_2) \Big\} .
   \end{equation}
   Here,   $W^\tau_k$  measures the work required for each species $k$  to move from state $\rho_k^{i-1}$ to state $\rho_k,$ in a  small time step $\tau>0.$  For PME nonlinear diffusion process,   the work in \cite{KM}  suggests to take  $W^\tau_k $   defined by the  $\mathcal W_2$-Wasserstein distance and $ \E(\rho )=\frac{1}{m-1}\int_\Omega  (\rho_1+\rho_2)^m ,$ where we assume to simplify the presentation that the drift is equal to $0.$ This approach leads to  interesting optimal transportation interpretation of the solution at the discrete time level  using Kantorovich potentials, and  transports maps  for each specie.     However, the continuum limit seems to be complex, and there does not seem to be a PDE approach to yield well-posedness on the continuum PDE, even if  the authors seem   to expect  to cover with this approach (at least for the particular situations   where the solutions are well-behaved  to stay  separated throughout the evolution, with stable interface in between them)  the solutions of the  cross-diffusion system :  
  \begin{equation}\label{CDOTM}
  	  \left\{ \begin{array}{ll} \partial_t{\rho_1}   -\frac{m}{m-1}\:  \nabla \cdot \left( \rho_1\:  \nabla  (\rho_1+\rho_2)^m \right)=0 \\  \\  
   \partial_t	{\rho_2}   -\frac{m}{m-1}\:   \nabla \cdot \left( \rho_2\:  \nabla     (\rho_1+\rho_2)^m \right)=0   
   \end{array}\right.  \quad \hbox{ in }(0,\infty)\times \Omega.  
  \end{equation} 
   	 	This system, subject to homogeneous Neumann boundary conditions and initial data, describes how both species avoid overcrowding regions through hyperbolic dynamics   in the spirit of  \eqref{crossgen2}.  	However, in this paper we deviate from the previous approach and proposes using the minimum flow problem to measure the work required for state transitions. We  consider:
    \begin{equation}\label{steepestdescent2}
   		\rho_i= \argmin_{\rho \in X} \Big\{    \tilde \E(\rho_1+\rho_2)     + \frac{1}{2\tau}\: I_1(\rho_1^{i-1}, \rho_1)+ \frac{1}{2\tau}\:  I_2(\rho_2^{i-1}, \rho_2) \Big\} ,
   	\end{equation}
   	 	where $I_k(\rho_k^{i-1}, \rho_k)$ minimizes the following cost function:
   	  \begin{equation}\label{balanceEq0}
   	\frac{1}{2} \int_\Omega \sigma_k\: \vert \omega\vert^2 -  \langle  \omega\cdot \nu,g_k\rangle_{\Gamma_{D_k}}   
   \end{equation}
   	 over $\omega\in L^2(\Omega)^N$ subject to the constraint:
\begin{equation}\label{div0}
	 -\nabla \cdot (\sigma_k\: \omega) = \rho_k^{i-1} - \rho_k, \quad \text{ in } \Omega, 
\end{equation}
   	 	with possible boundary conditions for $\sigma_k\: \omega\cdot \nu$ on $\partial\Omega\setminus \Gamma_{N_k}.$ 
   	The goal of $I_k(\rho_k^{i-1}, \rho_k)$   is to find the most efficient "traffic scheme" for transferring mass from $\rho_k^{i-1}$  to  $\rho_k.$  This scheme considers both the quadratic work of the flow $\omega$  between densities  $\rho_k^{i-1}$ and $ \rho_k,$  and the charge $g_k$ for $ \phi \cdot \nu$ at at specific boundary locations $  \Gamma_{D_k}.$
   	 Since the balance equation,  \eqref{balanceEq0} captures the essence of mass transfer, the transfer fee effectively models the diffusion process with boundary conditions, enabling analysis of overcrowding dynamics for two interacting species. Interestingly, unlike the single-species case, we'll see that taking for instance $\mathcal E(\rho)=\frac{1}{m+1}\int_\Omega ( \rho_1+\rho_2)^{m+1} $ this approach leads to a fundamentally different cross-diffusion system generalizing in some sense \eqref{CDOTM} (see Remark \ref{RspatialR}). This new system is parabolic in nature : 
   	 \begin{equation} 
   	 	\partial_t 	\rho _k- \nabla \cdot(\sigma_k\: \nabla \eta_k) = 0 ,  \quad  \hbox{ for }k=1,2,\quad \hbox{ in }(0,\infty)\times \Omega, 
   	 \end{equation} 
   	 where the potential  $\eta _k$ satisfies   the state equations 
   	  \begin{equation}  
   	 	\eta_k - (\rho_1+\rho_2)^m \in \partial I\!\! I_{[0,\infty)} (\rho_k), \quad  \hbox{ for each }k=1,2.    \end{equation} 
   	 Thanks to Remark \ref{remarqyeBeta}, this is   clearly connected to the maximization problem 
   	 	$$\max_{\rho_1,\rho_2\geq 0 } \Big\{  \rho_1\eta_1 + \rho_2\eta_2- \frac{1}{m+1}(  \rho_1  +\rho_2)^{m+1}\Big \}.$$ 
 Moreover,   the approach  incorporates mixed boundary condition for each specie, respect to the choice of the boundary condition in \eqref{balanceEq0} and \eqref{div0}.

   	\bigskip 
 In the following section, we precise the assumptions  and give the rigorous proofs for the results in a general case. Then, we show how one can use this approach to prove the results of Section \ref{Smain}. We show also how to connect the approach with gradient flow in dual Sobolev space in the homogeneous case.

   \subsection{Notations}
   
   For the sake of simplicity and ease of presentation,    we adopt the following   formal  notations: 
   \begin{itemize}
   	
   		\item[-]   $ [H^1_{\Gamma_D}(\Omega)]^2 := H^1_{\Gamma_{D_1}}(\Omega) \times H^1_{\Gamma_{D_2}}(\Omega) .$ 

   		\item[-]   $ [H^{-1}_{\Gamma_D}(\Omega)]^2 = H^{-1}_{\Gamma_{D_1}}(\Omega) \times H^{-1}_{\Gamma_{D_2}}(\Omega) .$  

   		\item[-]    $  [H^{1/2}(\Gamma_D) ]^2:=H^{1/2}(\Gamma_{D_1}) \times H^{1/2}(\Gamma_{D_2})  $ 
   		and $    [H^{1/2}_{00}(\Gamma_D) ]^2:=H^{1/2}_{00}(\Gamma_{D_1}) \times H^{1/2}_{00}(\Gamma_{D_2}) . $ 
   	Moreover, for any $\xi\in [H^1 (\Omega)]^2 $ and $g\in  [H^{1/2}(\Gamma_D)]^2,$ we say $\xi_{/\Gamma_{D}}=g,$   if and only if 
   	$$\xi_{1/\Gamma_{D_1}}=g_1\hbox{ and }\xi_{2/\Gamma_{D_2}}=g_2 .  $$

   		\item[-]  $[H^{-1/2}(\Gamma_{D})]^2:=    H^{-1/2}(\Gamma_{D_1})  \times H^{-1/2}(\Gamma_{D_2})  .$ 
   	Moreover,  for any  $h=(h_1,h_2)\in [H^{-1/2}(\Gamma_{D})]^2   $ and $g=(g_1,g_2)\in  [H^{1/2}_{00}(\Gamma_D) ]^2,  $        we use the notation   
   	$$  [h,g]_{\Gamma_D} \:   :=   \langle h_1,    g_1  \rangle_{ { H^{-1/2}(\Gamma_{D_1}),  H^{1/2}(\Gamma_{D_1})} }  +  \langle h_2,    g_2 \rangle_{ { H^{-1/2}(\Gamma_{D_2}),  H^{1/2}(\Gamma_{D_2})} }  . $$
   	 	Without abusing,  for any  $ \phi= (\phi _1,\phi_2)\in H_{div}(\Omega) \times H_{div}(\Omega) ,$  we'll use again the notation    $  [  \phi\cdot  \nu,g]_{\Gamma_D}    $ to   point out the sum of the duality products $<\phi_k\cdot \nu,g_k>$ on $\Gamma_{D_k}$  ; i.e.  
   	$$	\begin{array}{l}   
   		[\phi\cdot  \nu,g]_{\Gamma_D}  :=  \langle  \phi_1\cdot  \nu,g_1  \rangle_{ { H^{-1/2}(\Gamma_{D_1}),  H^{1/2}(\Gamma_{D_1})} }   +  \langle    \phi_2\cdot  \nu,g_2   \rangle_{ { H^{-1/2}(\Gamma_{D_2}),  H^{1/2}(\Gamma_{D_2})} } .	\end{array} $$
   	%
   	
   \end{itemize}

   At last, for any  $ \phi= ( \phi _1,\phi_2)\in L^2(\Omega)^N \times L^2(\Omega)^N$ and $\sigma =(\sigma_1,\sigma_2)\in [L^\infty(\Omega)]^2,$     we  use  the notations 
   $$\sigma\: \phi: =(\sigma_1\phi_1,\sigma_2\:\phi_2)  $$
   $$ [\nabla \cdot \sigma  \phi]  := (\nabla \cdot (\sigma_1\: \phi_1), \nabla \cdot (\sigma_2\: \phi_2 )),\quad \forall \:   \phi= ( \phi _1,\phi_2)\in L^2(\Omega)^N \times L^2(\Omega)^N$$  and   $$[\sigma\:\nabla \eta]:= (\sigma_1\:\nabla\eta_1,\sigma_2\: \nabla \eta_2),\quad \forall \eta=(\eta_1,\eta_2) \in [H^1(\Omega)]^2 .$$ 
    Then, we use the notations  $F_\sigma[\phi],$   for the quantity  
\begin{equation}\label{condF}
	  F_\sigma[\phi]= \frac{1}{2}\left( \sigma_1\:\vert \phi_1\vert^2  + \sigma_2 \vert \phi_2\vert^2\right),\quad \hbox{ for any }\phi\in [L^2(\Omega)^N]^2 
\end{equation}
   and   $\beta[\rho]$ for 
   $$\beta[\rho]:= \beta(.,\rho_1+\rho_2),\quad \hbox{ for any } \rho\in [L^2(\Omega)]^2.  $$

  \subsection{Proximal minimization for cross-diffusion modeling}\label{SectionProx}

   Let    
  \begin{equation} \label{hypggamma}
 \pi =(\pi_1,\pi_2)\in [H^{-1/2}(\Gamma_N)]^2\quad \hbox{ and }\quad g=(g_1,g_2)\in [H^{1/2}_{00}(\Gamma_D)]^2 , \end{equation} 
be given.      We denote by $\tilde g $ the  $[H^1 (\Omega)]^2$ function such that   
\begin{equation}\label{tildeg}
\tilde g =0 \hbox{ on }\Gamma_N\hbox{  and }  \tilde g =g  \hbox{ on }\Gamma_D.
\end{equation} 

\medskip 
 To define rigorously the transition work, 	for any   $\mu  \in [L^2(\Omega)]^2$  and $\chi   \in  [L^2(\Omega)^N]^2$,  we consider   
$$ \A_{\pi}^{\chi}[\mu]_k= \Big\{  \omega   \in    L^2(\Omega)^N,\: - \nabla \cdot (\sigma_k\: \omega ) =\mu_k +\nabla \cdot \chi_k  \hbox{ in }\Omega \hbox{ and }   (\sigma_k\omega  +\chi_k)\cdot \nu =\pi_k  \hbox{ on }\Gamma_{N_k }    \Big\} ,\quad   k=1,2. $$
That is, for any $k=1,2,$  $  \omega     \in \A_{\pi}^{\chi}[\mu]_k $ if and only if $\omega  \in   L^2(\Omega)^N$ and 
$$\int_\Omega (\sigma_k\: \omega +\chi_k)\cdot \nabla z \: dx =\int_\Omega \mu_k\:  z  \: dx   +\langle \pi_k,z\rangle_{\Gamma_{N_k}}\quad \hbox{ for any }z\in H^1_{\Gamma_{D_k  }}(\Omega). $$
In some sense, for each $k=1,2,$  $\A_{\pi}^{\chi}[\mu]_k$ is the set of all mass fluxes which balance  the mass  $\mu_k,$ under the action of the  external  force $\chi_k$ and the boundary action $\pi_k$ on $\Gamma_k.$   
Recall that  (see for instance \cite{IgGF}),   for any  $k=1,2,$  we have 
$$  \A_{\pi}^{\chi}[\mu]_k  \neq\emptyset,\quad \hbox{ for any }\mu \in  [L^2(\Omega)]^2  \hbox{ and }\chi\in   [L^2(\Omega)^N]^2  . $$

   \bigskip 
 Now, for any  $\mu=(\mu_1,\mu_2)\in [L^2(\Omega)]^2$ and  $\chi=(\chi_1,\chi_2) \in  [L^2(\Omega)^N]^2 ,$   we consider the  
 optimization problem 
 $$ \N_{\pi,g}^\chi [\mu]   :=    	 	\inf_{\rho,\phi}\left \{\int_\Omega   \beta[\rho] +  \int_\Omega F_\sigma[ \phi ]  -   [( \sigma\: \phi+\chi)\cdot  \nu, g]_{\Gamma_D}  \:  :\:     \rho\in [L^2_+(\Omega)]^2    ,\:   \phi \in     \A_{\pi}^{\chi}[\mu-\rho]   \right\},  $$    
 where 
 $$   \A_{\pi}^{\chi}[\mu-\rho]:=  \A_{\pi}^{\chi}[\mu-\rho]_1\times    \A_{\pi}^{\chi}[\mu-\rho]_2  .$$
%
  See that $ \N_{\pi,g}^\chi [\mu]$ may be written as 
  $$ \N_{\pi,g}^\chi [\mu]  = \inf_{\rho,\omega}\left \{\int_\Omega   \beta[\rho] + \underbrace{	\I_{\pi,g}^{\chi}[\mu-\rho]_1 +	\I_{\pi,g}^{\chi}[\mu-\rho]_2  }_{	\I_{\pi,g}^{\chi}[\mu-\rho]}  \: :\:  \rho\in [L^2_+(\Omega)]^2\right\},  $$    
 where  
 \begin{equation}\label{Imu}
 	\I_{\pi,g}^{\chi}[\mu-\rho]_k := 	\inf_{\omega}\left \{   \frac{1}{2}\int_\Omega \sigma_k\: \vert   \omega   \vert^2   -\langle (\sigma\: \omega+\chi_k)\cdot \nu, g_k\rangle_{\Gamma_{D_k}} \:  :\:    \omega\in  \A_{\pi}^\chi[\mu-\rho]_k   \right\}.
 \end{equation}
 Here the quantity  $ \I_{\pi,g}^\chi[\mu-\rho] $ represents   the transition work associated  with the distribution of mass  $\mu-\rho$, and the term $\nabla \cdot \chi $  can be utilized to represent a significant external influence exerted on the system by an external force, $\chi$ (along with boundary conditions).

  \medskip

 Thanks to  \cite{IgGF}, remember that 
  \begin{proposition} (cf. \cite{IgGF}) \label{Lduals} 
For any $k=1,2,$ we have  
  	\begin{equation}\label{std}
  		\I_{\pi,g}^{\chi}[\mu]_k 
  		= 	\max_{z\in H^1(\Omega),\: z_{/\Gamma_{D _k}}=g _k }\left \{   \int_\Omega  \mu_k\: z    -  \int_\Omega \chi_k \cdot \nabla  z    - \frac{1}{2}   \int_\Omega \sigma_k\:  \vert \nabla z\vert^2    +  \langle 	\pi_k,z \rangle_{\Gamma_{N_k}}   \right\}   .  
  \end{equation}
  Moreover, $ \omega$ and $z $ are optimal if  and only if $\omega=\nabla z$ and $z$  is the unique solution of   the following PDE 
  \begin{equation}
  	\label{system1}
  	\left\{
  	\begin{array}{ll}
  		-\nabla\cdot(\sigma_k\: \nabla  z)  =\mu_k -\nabla \cdot \chi_k   
  		\quad & \hbox{ in } \Omega,\\  \\
  		(\sigma_k\:\nabla z +\chi_k)\cdot \nu  =  \pi_k & \hbox{ on }\Gamma_{N _k}\\  \\
  		z  =  g_k ,\quad    & \hbox{ on }\Gamma_{D _k}.
  	\end{array}
  	\right.
  \end{equation} 
  \end{proposition}

 Coming back to the problem $ \N_{\pi,g}^\chi [\mu],$ we have the following duality result. 
 \begin{theorem}	\label{tduality1}
 	For any $\mu=(\mu_1,\mu_2)\in [L^2(\Omega)]^2$ and $\chi=(\chi_1,\chi_2)\in  [L^2(\Omega)^N]^2 ,$   we have  :  
\begin{equation}
 \N_{\pi,g}^\chi [\mu]  = 	\underbrace{ \max_{\eta\in [H^1(\Omega)]^2,\eta_{/\Gamma_D}=g} \left\{  \int_\Omega  \mu\cdot\eta+\int_\Omega [\chi\cdot\nabla\eta]  -   \int_\Omega   F_\sigma[\nabla\eta]   -\int_{\Omega } \beta^*[\eta]   +  [\pi,\eta]_{\Gamma_N}  \right\}}_{ =:D_{\pi,g}^{\chi}[\mu]} ,
\end{equation}     
where                                                                                                                              
$$\beta^*[\eta]:= \beta^*\left(x, \eta_1\vee \eta_2\right),\quad \hbox{ for any }\eta \in [L^2(\Omega)]^2.$$  
 	Moreover, $(\rho,\phi)$ and $\eta$ are solutions of  $ \N_{\pi,g}^\chi [\mu]  $ and $ D_{\pi,g}^{\chi}[\mu] ,$ respectively,  if and only if     $\phi_k= \nabla \eta_k $ and the couple $(\rho_k,\eta_k)$ satisfies the following system of PDE :  
 		\begin{equation}              	\label{PDE1}                                                                                                                    
 		\left\{\begin{array}{ll}                                                                                                 	\left.  \begin{array}{ll}                                                                                                      
 				\rho _k- \nabla \cdot(\sigma_k\:   \nabla \eta_k) =  \mu _k +\nabla \cdot \chi_k  \\               
 				\\  			   \eta_k \in \partial\beta(.,\rho_1+\rho_2)+\partial I\!\!I_{[0,\infty)}(\rho_k) 
 			\end{array}\right\} 	\quad & \hbox{ in } \Omega, \\  \\                                                                        
 			(\sigma_k\: \nabla\eta_k+\chi_k)  \cdot \nu  =  \pi_k   & \hbox{ on }\Gamma_{N _k}\\  \\                                        
 			\eta_k  =  g _k \quad    & \hbox{ on }\Gamma_{D _k} ,   	\end{array}  		\right. 	  \quad  \hbox{ for }k=1,2                                                                                                                     
 	\end{equation}  
 	in the sense that $\rho _k\in L^2_+(\Omega),$  $\eta_k\in H^1 (\Omega),$  $\eta_k=g_k$ on $\Gamma_{D _k},$  
 	$$ \eta_1\vee\eta_2 =:\tilde\eta \in \partial \beta(x,   \rho_1+ \rho_2),\: \quad  \eta_k -\tilde \eta    \in  \partial  I\!\! I_{[0,\infty)}(\rho_k),   \quad \hbox{ a.e.  in }\Omega,$$ 
 	and 
 			$$\int_\Omega  \rho _k\: \xi_k   + \int_\Omega(\sigma_k\:  \nabla\eta_k+\chi_k) \cdot \nabla \xi_k  =\int_\Omega \mu _k\: \xi_k     +\langle \pi_k, \xi_k \rangle_{\Gamma_k} , \quad \forall\:    \xi_k\in  H^1_{\Gamma_{_k} } (\Omega) .$$
  \end{theorem}

  \bigskip 
To the proof of this theorem, we see first that we have. 
  
 \begin{proposition}\label{Pexistprimal}
 	For any $\mu=(\mu_1,\mu_2)\in [L^2(\Omega)]^2$ and $\chi\in  [L^2(\Omega)^N]^2 ,$    the problems $ \N_{\pi,g}^\chi [\mu]  $ and $ D_{\pi,g}^{\chi}[\mu] $ have  solutions. 
 \end{proposition}
\begin{proof}
	Let $(\rho^n,\phi^n)$ be a minimizing sequence $ \N_{\pi,g}^\chi [\mu]  .$  Thanks to the assumptions on   $\beta$,  it is clear that $\rho^n $ and    $\phi^n$  are bounded in $[L^2_+(\Omega)]^2 $   and $[L^2(\Omega)^N]^2$,   respectively. So, there exists a subsequence that we denote again by  $(\rho^n,\phi^n )\in   [L^2_+(\Omega)]^2\times [L^2(\Omega)^N]^2$, and $(\rho,\phi)\in   [L^2_+(\Omega)]^2\times [L^2(\Omega)^N]^2,$  such that 
	$$\rho^n\to \rho,\quad \hbox{ in } [L^2(\Omega)]^2 \hbox{-weak}^*$$ 
	and   $$\phi^n\to \phi,\quad \hbox{ in } [L^2(\Omega)^N]^2 \hbox{-weak} .$$ 
	Combining, in addition,  \eqref{tracexpression} with the fact that   $-[\nabla \cdot (\sigma\:\phi^n+\chi)]=\mu-\rho^n$ and $(\sigma\:\phi^n+\chi)\cdot \nu=\pi$ on $\Gamma_N,$ we see that 
	$$   [(\sigma\: \phi^n+\chi)\cdot \nu,g]_{\Gamma_D}\to   [(\sigma\:\phi+\chi)\cdot \nu ,g]_{\Gamma_D}. $$  
	Clearly $(\rho,\phi )$ is an admissible test function for the optimization problem $ \N_{\pi,g}^\chi [\mu]  $ ; i.e. $(\rho,\phi )\in \A_{\pi}^\chi [\mu-\rho]$.  Then, using the lsc and convexity of $\beta $ and $F_\sigma $, we deduce that  $(\rho,\phi )$ is a solution of the   optimization problem  $ \N_{\pi,g}^\chi [\mu]  .$   The proof for 
	$ D_{\pi,g}^{\chi}[\mu] $ follows more or less the same ideas.  \end{proof}
   
 \bigskip 
To prove Theorem \ref{tduality1}, let us consider the application $K  \: :\:   [H^{-1}_{\Gamma_D} (\Omega)]^2   \to \RR$ given by 
$$ K[f]  :=    	 \int_\Omega f_0\cdot   \tilde g -\int_\Omega [\overline f \cdot \nabla \tilde g] + 	\inf_{\rho,\phi }\left \{   \int_\Omega  \beta[\rho] +  \int_\Omega   F_\sigma[\phi] \: dx -  [ (\sigma\:\phi+\chi+ \overline f )\cdot  \nu, g]_{\Gamma_D}  \:\right.$$  $$\left. :\:   \rho\in  [L^2_+(\Omega)]^2   ,\:  \phi\in     \A_{\pi}^{\chi+\overline f}[\mu+f_{0}-\rho]    \right\},  $$
where   $f_0=(f_{01},f_{02})$ and $\overline f=(\overline f_1,\overline f_2)$ are given by the decomposition  $f_k=f_{0k}+\nabla \cdot \overline f_k$ in $H^{-1}_{\Gamma_{D_k}}(\Omega),$ for each $k=1,2.$   
 Then, since $ \N_{\pi,g}^\chi [\mu]  = K[0],$ it is enough to prove that $K[0]$ coincides with $D_{\pi,g}^{\chi}[\mu].$ To this aim we use duality techniques which involves  the Legendre transforms of  $K.$  This is the aim of the following lemma.

 \begin{remark} Thank to    \eqref{tracexpression},  we see  that  $K$ may be rewritten by using $\tilde g$ given by \eqref{tildeg}. Indeed, 	 thanks to \eqref{tildeg}, we have 	$$ [  (\sigma\:\phi   +\chi+ \overline f )\cdot  \nu, g]_{\Gamma_D}  
 	= \int_\Omega [(\sigma\:\phi+\chi+ \overline f)  \cdot \nabla \tilde g] -  \int_\Omega  (f_{0}+\mu-  \rho)\cdot  \tilde g     ,$$  
 	for any  $ \rho\in  [L^2_+(\Omega)]^2$ and $  \phi  \in  \A_{\pi}^{\chi+\overline f}[\mu+f_0-  \:  \rho].$ This implies   that  $K $ may be written as 
 	\begin{equation}\label{mofK}
 		\begin{array}{c} 
 			K     [f]   =    	 	\inf_{\rho,\phi }\left \{   \int_\Omega  \beta[\rho] +  \int_\Omega   F_\sigma[\phi] \: dx   - \int_\Omega [ (\sigma\:\phi   +\chi )  \cdot \nabla \tilde g] + \int_\Omega  (\mu-  \rho) \cdot \tilde g  \: \right.  \\  \\  \: 
 			\left. :\:   \rho\in  [L^2_+(\Omega)]^2,\:    \phi  \in  \A_{\pi}^{\chi+\overline f}[\mu+f_0- \rho]    \right\}. 
 	\end{array}  \end{equation} 
 	
 \end{remark}

\begin{lemma}\label{Kconvecl.s.c.}
We have 
	\begin{equation} \label{perturbduality}
		-K  [0] = -K  ^{**}[0] = \min_{ \eta \in [H^1_{\Gamma_D}(\Omega)]^2} K  ^{*}[\eta] ,
	\end{equation} 
\end{lemma}
\begin{proof}To prove  \eqref{perturbduality}, it is enough to prove that  	$K  $ is convex and l.s.c.  and      conclude   by classical duality results  (cf.  \cite{Ekeland}).  \\	
	\underline{Convexity :} For any $f,\: h\in [H^{-1} _{\Gamma_D}(\Omega)]^2,$ taking $(\rho_f,\phi _f)$ and $(\rho_h,\phi _h)$ the solutions corresponding to the optimization problems $K  [f]  $ and 	$K  [h]  $ respectively, one sees that  $$\underbrace{ t\phi _h+(1-t)\phi _f}_{\phi _t}\in \A_{\pi}^{\overline h_{t}}[\mu-\alpha\underbrace{(t\rho_h+(1-t)\rho_f}_{\rho_t}+\underbrace{(th_0+(1-t)f_0)}_{h_{0t}} ],$$ where 
	$\overline h_{t}= t\overline h+(1-t)\overline f,$ for any $t\in [0,1].$   Moreover, using the convexity of $\beta (x,.)$ and $F_\sigma,$ we have 
	$$ \int\beta[\rho_t]	+ \:  \int_\Omega F_\sigma[\phi _t] -	   [ (\sigma\:\phi  _t+\chi+\overline f_t)\cdot \nu,g]_{\Gamma_D}     \leq t \left(   \int_\Omega \beta[\rho_h]	+ \:  \int_\Omega F_\sigma[\phi _h]-[\phi _h+\chi+\overline h) \cdot \nu,g]_{\Gamma_D}  \right)$$ 
	$$ +  (1-t)\left(   \int_\Omega\beta[\rho_f]	+ \:  \int_\Omega F_\sigma[\phi_f] -[ (\sigma\:\phi  _f+\chi+\overline f)\cdot \nu,g]_{\Gamma_D} \right),$$ 
	which implies that 
	$$    K  [th+(1-t)f]  \leq  t K  [h]+(1-t)K  [f]  . $$
	\underline{Lower semi-continuity :}  Let us consider $f^n$ a sequence of $[H^{-1}_{\Gamma_D}(\Omega)]^2$ which converges to $f.$ That is a sequence of $[L^2(\Omega)]$ and $[L^2(\Omega)^N]$ functions $f_0^n$ and $\overline f^n$ respectively, such that 
	$$ [f^n,\xi]_\Omega = \int_\Omega f_0^n\cdot \xi -\int_\Omega [\overline f^n\cdot \nabla \xi] \to   \int_\Omega f_{0}\cdot \xi -\int_\Omega [\overline f\cdot \nabla \xi]  = [f,\xi]_\Omega  , \quad \forall \xi\in [H^1_{\Gamma_D}(\Omega)]^2.$$
	Let us prove that $K[f] \leq \liminf_{n\to\infty } K[f^n].$   To this aim, we  consider  $(\rho^n,\phi ^n)$ be  the solution corresponding to $K[f^n]$ ; i.e.  
	\begin{equation}\label{testphin}
		\phi ^n \in  \A_{\pi}^{\chi+\overline f^n}[\mu+f_0^n-  \:  \rho^n].  
	\end{equation}
	and
	$$	 K[f^n]   =    	    \int_\Omega   \beta[\rho^n] +  \int_\Omega F_\sigma[\phi ^n] \: dx   -\int_\Omega [ (\sigma\:\phi  ^n+\chi )  \cdot \nabla \tilde g  ]+ \int_\Omega  (\mu-  \rho^n)\cdot \tilde g, $$
	where we use \eqref{mofK}.    We can assume that $  K[f^n] $ is bounded. So,  there exists $C<\infty $ such that   
	\begin{equation} 
		\int_\Omega   \beta[\rho^n] +  \int_\Omega F_\sigma[\phi ^n] \: dx  \leq C  +\int_\Omega [ (\sigma\:\phi  ^n+\chi )  \cdot \nabla \tilde g  ]- \int_\Omega  (\mu-  \rho^n)\cdot \tilde g
	\end{equation} 
	Using  assumptions  $(H1)$ and $(H2)$ with Young formula,  we see that   $  \rho^n $ and    $\phi^n$  are bounded in $ [L^2_+(\Omega)]^2  $   and $  [L^2(\Omega)^N]^2$,   respectively. So, there exists $(\rho,\phi)\in   [L^2_+(\Omega)]^2\times [L^2(\Omega)^N]^2 $  and a sub-sequence that we denote again by  $(\rho^n,\phi^n )  $,  such that 
	$$\rho^n\to \rho,\quad \hbox{ in } [L^2(\Omega)]^2 \hbox{-weak}$$ 
	and   $$\phi^n\to \phi,\quad \hbox{ in }  [L^2(\Omega)^N]^2 \hbox{-weak} .$$ 
	Using \eqref{testphin}, we have 
	\begin{equation}
		\int_\Omega [ (\sigma\:\phi ^n-\chi)\cdot \nabla \xi ]= \int_\Omega (\mu-\rho^n)\cdot  \xi +[f^n,\xi]_{\Omega} +[ \pi,\xi]_{\Gamma_N},\quad \forall \xi\in [H^1_{\Gamma_D}(\Omega)]^2,   
	\end{equation}
	and, by  letting $n\to \infty,$ we get 
	\begin{equation}
		\int_\Omega [ (\sigma\:\phi -\chi)\cdot \nabla \xi] = \int_\Omega (\mu-\rho )\cdot \xi +[ f,\xi]_{\Omega}+[ \pi,\xi]_{\Gamma_N},  \quad \forall \xi\in [H^1_{\Gamma_D}(\Omega)]^2. 
	\end{equation} 
	This implies that  $  \phi  \in  \A_{\pi}^{\chi+\overline f}[\mu+f_0-  \:  \rho] .$      Then, using the l.s.c. and convexity of $\beta $ and $F  ,$   we have 
	\begin{eqnarray*}
		K[f]&\leq &    \int_\Omega  \beta[\rho] +  \int_\Omega   F_\sigma[\phi]    -\int_\Omega [(\sigma\:\phi+\chi )  \cdot \nabla \tilde g]  + \int_\Omega  (\mu-  \rho) \cdot  \tilde g   \\  \\ 
		& \leq &    \liminf_{n\to\infty } \left\{   \int_\Omega   \beta[\rho^n] +  \int_\Omega F_\sigma[\phi ^n]     -\int_\Omega [ (\sigma\:\phi  ^n+\chi )  \cdot \nabla \tilde g]  +  \int_\Omega  (\mu-  \rho^n) \cdot \tilde g \right\}  \\  \\  &=&      \liminf_{n\to\infty }  K[f^n].\end{eqnarray*} 
	Thus the result.

\end{proof}

\medskip 
\begin{proofth}{Proof of Theorem \ref{tduality1}}   
We only need to compute 
	$K_1^{*}.$      By definition,  for any $\eta \in [H^1_{\Gamma_D}(\Omega)]^2,$  we have

	$$ K^*[\eta] =\max\left\{  [ f,\eta ]_{\Omega} -  K[f ] \: :\:   f\in [H^1_{\Gamma_D}(\Omega)]^2\right\}    $$
	$$ = \max_{f_0,\overline f,\phi  }\Big\{  \underbrace{\int_\Omega  f_0\cdot\eta  -\int _\Omega [\overline f\cdot\nabla\eta]  - \int_\Omega \beta[\rho] 	- \:  \int_\Omega F_\sigma[\phi]  +\int_\Omega [ (\sigma\:\phi    +\chi )  \cdot \nabla \tilde g ] -\int_\Omega  ( \mu -\rho) \cdot \tilde g  }_{H[f_0,\overline f,\rho,\phi  ]}  $$  
	$$    \: :\:  (f_0,\overline f)\in [L^2(\Omega)]^2\times [L^2(\Omega)^N]^2,\: \phi  \in \A_{\pi}^{\chi+\overline f} [\mu-\rho+f_0],\: \rho\in [L^2_+(\Omega)]^2\Big\}.    $$
	Using \eqref{tracexpression}, we have 
	$$\int_\Omega  f_0\cdot\eta  -\int _\Omega [\overline f\cdot\nabla\eta] = \int_\Omega [(\sigma\:\phi+\chi)\cdot \nabla \eta ] -\int_\Omega (\mu-\rho)\cdot \eta -[\pi,\eta]_{\Gamma_N},$$ 
so that 
	\begin{eqnarray*} 
		H[f_0,\overline f,\rho,\phi  ] 
		&=& \int_\Omega ([\sigma\:\phi\cdot \nabla (\eta+\tilde g)]-F_\sigma[\phi]  ) +\int_\Omega (\rho\cdot (\eta+\tilde g) - \beta[\rho] ) \\  \\ 
		&  & - \int_\Omega \mu\cdot(\eta+\tilde g)   + \int_\Omega [\chi\cdot \nabla (\eta+\tilde g) ] -[\pi,\eta]_{\Gamma_N}. \end{eqnarray*}
This implies  that 
	\begin{eqnarray*}
		K^*[\eta] 
		& =& 	\max_{\rho,\phi  }\left\{  \int_\Omega ([\sigma\:\phi\cdot \nabla (\eta+\tilde g)]-F_\sigma[\phi]  ) +\int_\Omega (\rho\cdot (\eta+\tilde g) - \beta[\rho] ) - \int_\Omega \mu\cdot(\eta+\tilde g)] \right.  \\  \\ &  & \hspace*{3cm}\left.   + \int_\Omega \chi\cdot \nabla (\eta+\tilde g)   -[ \pi,\eta]_{\Gamma_N}      \: :\:    (\rho,\phi  )\in [L^2_+(\Omega)]^2\times [L^2(\Omega)^N]^2 \right\}   \\  \\ 
		&=&  \max_{\phi   \in [L^2(\Omega)^N]^2  }  \int_\Omega ([\sigma\:\phi\cdot \nabla (\eta+\tilde g)]-F_\sigma[\phi]  ) +  \max_{\rho \in [L^2_+(\Omega)]^2}  \int_\Omega (\rho\cdot (\eta+\tilde g) - \beta[\rho] )    \\  \\ &  & \left. \hspace*{3cm}- \int_\Omega \mu\cdot(\eta+\tilde g)   + \int_\Omega  [\chi\cdot \nabla (\eta+\tilde g)]   -[ \pi,\eta]_{\Gamma_N}  \right\}   .
	\end{eqnarray*}
	Using the assumptions $(H1)$ and $(H2)$, we deduce    that  
	\begin{eqnarray*}
		K^*[\eta]&= & \int_\Omega F_\sigma[\nabla (\eta+\tilde g)]   +\int_\Omega \beta^*[\eta+\tilde g]  - \int_\Omega \mu\cdot(\eta+\tilde g)   + \int_\Omega [\chi\cdot \nabla (\eta+\tilde g)]   -[ \pi,\eta]_{\Gamma_N} .
	\end{eqnarray*}
	Combining this with \eqref{perturbduality}, we obtain 
	\begin{eqnarray*}
		-K[0]&=&  
		\min_{ \eta \in [H^1(\Omega)]^2} \int_\Omega F_\sigma[\nabla (\eta+\tilde g)]    +\int_\Omega \beta^*[\eta+\tilde g]  - \int_\Omega \mu\cdot(\eta+\tilde g)   + \int_\Omega [\chi\cdot \nabla (\eta+\tilde g) ]  -[ \pi,\eta]_{\Gamma_N} \\  \\
		&=& \min_{ \eta \in [H^1(\Omega)]^2,\: \eta_{/\Gamma_D}=g}   \int_\Omega   F_\sigma[\nabla \eta]  +\int_\Omega  \beta^*[\eta] - \int_\Omega \mu\cdot \eta    + \int [\chi\cdot \nabla \eta ]   -[ \pi,\eta]_{\Gamma_N}. 
	\end{eqnarray*}
	Thus the duality  $\N_{\pi,g}^\chi (\mu) = D_{\pi,g}^\chi( \mu) . $ 
	Now, thanks to Lemma \ref{Pexistprimal}, let us consider $(\rho,\phi  )$ and $\eta$ be the solutions of $\N_{\pi,g}^\chi [\mu] =D_{\pi,g}^{\chi}[\mu] .$  We have 
	$$   \int_\Omega  \mu\cdot\eta-\int_\Omega [\chi\cdot \nabla\eta]  -\int_\Omega F_\sigma[\nabla \eta ]    -\int_{\Omega }  \beta^*[\eta]   + [ \pi,\eta]_{\Gamma_N}  $$ 
	$$=  \int_\Omega   \beta[\rho] +  \int_\Omega F_\sigma [\phi]  -  [  (\sigma\:\phi   +\chi)\cdot  \nu, g]_{\Gamma_D} .$$ 
	Since $ \phi   \in \A_{\pi}^\chi [\mu-\rho] ,$ $\eta \in [H^1(\Omega)]^2$ and $\eta =g$ on $\Gamma_D,$  we also have   
	$$\int_\Omega  \rho\cdot \eta   + \int_\Omega[(\sigma\:\phi+ \chi)   \cdot \nabla \eta]  =\int_\Omega \mu\cdot \eta    +[ \pi , \eta ]_{\Gamma_N}  + [  (\sigma\:\phi   +\chi)\cdot \nu, g ]_{\Gamma_D}   . $$ 
	Combining both equation, we get 
	$$\int_\Omega F_\sigma[\phi]     + \int_\Omega F_\sigma[\nabla  \eta ]   -\int_\Omega [\sigma\:\phi\cdot \nabla \eta]     =  \int_\Omega \eta\cdot  \rho -  \int_{\Omega }  \beta^*[\eta]  -  \int_{\Omega } \beta [\rho] . $$
 	Then using moreover the fact that $ [\phi\cdot \nabla (\eta+\tilde g)] \leq F_\sigma[\nabla (\eta+\tilde g)]+ F_\sigma[\phi] $
 	and $\rho\cdot \eta \leq \beta^*[\eta]  +    \beta [\rho],$ a.e. in $\Omega,$ we deduce that 
 	  $ \int_\Omega F_\sigma[\nabla (\eta+\tilde g)]     + \int_\Omega F_\sigma[\phi]     =   [\phi\cdot  \nabla \eta]$ and $ \beta^*[\eta]  +    \beta [\rho] =\rho\cdot \eta$ ,     a.e. in $\Omega,$ for each $k=1,2.$  Thus   $(\rho,\eta,\phi  )$ is a solution of the PDE \eqref{PDE2}.  For the the  proof of the converse part, one sees first directly that $D_{\pi,g}^{\chi}[\mu] \leq \N_{\pi,g}^\chi[\mu].$  Then, by working with the solution of  \eqref{PDE2} one proves     $D_{\pi,g}^{\chi}[\mu] = \N_{\pi,g}^\chi[\mu] .$

\end{proofth}

\bigskip 
To end up this section, we give the following result  which is a direct consequence of Theorem \ref{tduality1} and which will be useful for the sequel

\begin{corollary}\label{CorEst}
	For any  $\mu=(\mu_1,\mu_2)\in [L^2(\Omega)]^2$, $\pi =(\pi_1,\pi_2)\in [H^{-1/2}(\Gamma_N)]$  and $g =(g_1,g_2)\in [H^{1/2}_{00}(\Gamma_D)]$   the system of PDE 
		\begin{equation}
		\label{PDE2}
		\left\{\begin{array}{ll} 
			\left.  \begin{array}{ll} 
				\rho _k-\nabla \cdot(\sigma_k\:   \nabla \eta_k)  =  \mu _k -\nabla \cdot \chi_k     \\ 
				\\  			   \eta_k \in \partial\beta(.,\rho_1+\rho_2)+\partial I\!\!I_{[0,\infty)}(\rho_k)   
			\end{array}\right\} 	\quad & \hbox{ in } \Omega, \\  \\
			(\sigma_k\: \nabla\eta_k+\chi_k)  \cdot \nu  =  \pi_k   & \hbox{ on }\Gamma_{N _k}\\  \\
			\eta_k  =  g _k ,\quad    & \hbox{ on }\Gamma_{D _k}
		\end{array}
		\right\} \quad  \hbox{ for }k=1,2, 	 
	\end{equation} 
	has a solution $(\rho,\eta)$ in  the sense that, for each $k=1,2, $    	$\rho _k\in L^2_+(\Omega),$  $\eta_k\in H^1 (\Omega),$  $\eta_k=g_k$ on $\Gamma_{D _k},$  
	$$ \eta_1\vee\eta_2 =:\tilde\eta \in \partial \beta(x,   \rho_1+ \rho_2),\: \quad  \eta_k -\tilde \eta    \in  \partial  I\!\! I_{[0,\infty)}(\rho_k),   \quad \hbox{ a.e.  in }\Omega,$$ 
	and 
	$$\int_\Omega  \rho _k\: \xi_k   + \int_\Omega(\sigma_k\:  \nabla\eta_k+\chi_k) \cdot \nabla \xi_k  =\int_\Omega \mu _k\: \xi_k     +\langle \pi_k, \xi_k \rangle_{\Gamma_k} , \quad \forall\:    \xi_k\in  H^1_{\Gamma_{_k} } (\Omega) .$$

	 $$\int_\Omega  \rho _k\: \xi _k    + \int_\Omega \sigma_k\:  \nabla\eta_k\cdot \nabla \xi_k  =\int_\Omega \mu _k\: \xi _k   -\int \chi_k\cdot \nabla \xi_k    +\langle  \pi_k, \xi_k\rangle_{\Gamma_{N_k}}  ,\quad \forall \:  \xi\in H^1_{\Gamma_{D_1} } (\Omega) \times  H^1_{\Gamma_{D_2} } (\Omega). $$ 
	 In particular, for each $k=1,2, $ we have 
	 	\begin{equation}\label{eqextreme}
	 	\int_\Omega  \rho_k\: (\eta_k-\tilde g_k)    + \int_\Omega \sigma_k\:  \nabla\eta_k \cdot \nabla(\eta_k-\tilde g_k)   =\int_\Omega \mu_k  \: (\eta_k-\tilde g_k)     -\int \chi_k\cdot \nabla (\eta_k-\tilde g_k)   +\langle \pi_k , (\eta_k-\tilde g_k) \rangle_{\Gamma_{N }}   . 	\end{equation}	 
	 

\end{corollary}
%
%
%
%
%
%

 \begin{remark}
 \begin{enumerate}
 	\item 	For simplicity, we present in this paper the case where $F_\sigma$ is is quadratic as defined  in  \eqref{condF}. However, the framework can be extended to a more general form: $F [\Phi] = F_1(x,\phi_1)+ F_2(x,\phi_2),$ where each $F_k$ is tailored to capture the specific diffusion behavior  of     the specie $k=1,2.$  Readers interested in exploring   the expected associated dynamic  for more general   $F $  can refer to \cite{IgGF}.
 	
 	\item While assumption \eqref{minsigma} suffices for the proofs in this paper, exploring specific scenarios where either $\sigma_1\equiv 0$ or $\sigma_2\equiv 0$  holds significant interest for applications.  We believe that some of the results might still be valid under the assumptions $\min_{x\in \overline \Omega} \sigma_1(x) \: \min_{x\in \overline \Omega} \sigma_2(x) \neq 0 $, potentially requiring additional compatibility conditions.  We defer the details of this case for future work focused on special applications. 
 \end{enumerate}

 \end{remark}

\section{Existence for evolution problem}  
   \label{Sexistence}

Our objective here is to establish the existence of a weak solution. As mentioned in the introduction,  we proceed by discretizing time with discrete steps of size $\tau$, denoted as $t_0 = 0 < t_1 < \cdots < t_n <T$, and then consider the limit as $\tau$ approaches zero in the sequence of piecewise constant curves 
$$\rho^\tau =\sum_{i=0}^{n-1}\rho^i\:  \chi_{[t _k,t_{i+1}[}  +   \rho_0\:  \chi_{[-\tau,0[}   ,$$
with $\rho^0=\rho_0.$  Here $\rho^i$ is given   by 
 $$\rho^i = 	 	\hbox{argmin}_{\rho}\left \{\int_\Omega   \beta[\rho] +  \tau\: \int_\Omega F_\sigma[  \phi ] \: dx - \tau\:     [ (\sigma\:\phi +\rho^{i-1}V^i+\overline f^i)\cdot  \nu, g]_{\Gamma_D} \right. $$ $$\left.  \: :\:    (\rho,\phi)\in \A_{\pi}^{\rho^{i-1}V^i+\overline f^i}[(\tau f_0^i +\rho^{i-1}-\rho)/\tau ],\:  \rho\in [L^2_+(\Omega)]^2 \right\} $$    
 and  
 $$  f_0^i(.) :=  \left( \frac{1}{\tau} \int_{t_{i-1}}^{t_i} f_{01}(t,.)\: dt ,  \frac{1}{\tau} \int_{t_{i-1}}^{t_i} f_{02}(t,.)\: dt   \right)\hbox{ and } \overline f^i(.) :=  \left( \frac{1}{\tau} \int_{t_{i-1}}^{t_i} \overline f_{1}(t,.)\: dt ,  \frac{1}{\tau} \int_{t_{i-1}}^{t_i} \overline f_{2}(t,.)\: dt   \right),  $$
 In other words, $\rho^i$ is given by 
  $$\rho^i = 	 	\hbox{argmin}_{\rho}\left \{\int_\Omega   \beta[\rho] +  \tau\:   \I_{\pi,g}^{\rho^{i-1}V^i+\overline f^i}[(\tau f_0^i +\rho^{i-1}-\rho)/\tau ]\:  :\:   \rho\in [L^2_+(\Omega)]^2 \right\}. $$     
  Thanks to Theorem \ref{tduality1}, we know that $\rho_i$ satisfies the following system  PDE 
	\begin{equation}
	\label{PDEi}
	\left\{\begin{array}{ll} 
		\left.  \begin{array}{ll} 
			\rho _k^i-\tau\: \nabla\cdot (\sigma_k\:  \nabla \eta _k^i+\rho^{i-1}_kV_k^i  +\overline f_{k}^i  ) = \tau f_{0k}^i   + \rho_k^{i-1} ,    \\ 
			\\  	 \eta_k^i \in \partial\beta(.,\rho_1^i+\rho_2^i)+\partial I\!\!I_{[0,\infty)}(\rho_k^i) 		 
		\end{array}\right\} 	\quad & \hbox{ in } \Omega, \\  \\
		(\sigma_k\:\nabla\eta_k^i+\rho^{i-1}_kV_k^i+\overline f_{k}^i  ) \cdot \nu  =  \pi_k   & \hbox{ on }\Gamma_{N _k}\\  \\
		\eta_k^i  =  g _k ,\quad    & \hbox{ on }\Gamma_{D _k}
	\end{array}
	\right\} \quad  \hbox{ for }k=1,2, 	 
\end{equation} 
in the sense that, for each  $k=1,2,$ for each $k=1,2, $    	$\rho _k^i\in L^2_+(\Omega),$  $\eta_k^i\in H^1 (\Omega),$  $\eta_k^i=g_k$ on $\Gamma_{D _k},$  
$$ \eta_1^i\vee\eta_2^i =:\tilde\eta^i \in \partial \beta(x,   \rho_1^i+ \rho_2^i),\: \quad  \eta_k^i -\tilde \eta    \in  \partial  I\!\! I_{[0,\infty)}(\rho_k^i),   \quad \hbox{ a.e.  in }\Omega,$$ 
and  	$$\int_\Omega ( \rho _k^i-\rho _k^{i-1})\: \xi_k  \: dx+\tau\:  \int_\Omega (\sigma_k\:  \nabla\eta _k^i+ \rho^{i-1}_kV_k^i+\overline f_{k}^i  ) \cdot \nabla \xi_k  \: dx =\tau\:  \int_\Omega f^i_{0k}\: \xi _k  \: dx + \tau\: \langle \pi_k^i , \xi _k\rangle_{\Gamma_{N_{k}}} , $$	
		 for any $\xi\in H^1_{\Gamma_{D_1} } (\Omega) \times  H^1_{\Gamma_{D_2} } (\Omega).$ 	

 \medskip 
    Let us define the sequence 
 $$\tilde \rho^{\tau}(t) =\frac{(t-t_i)\rho^{i+1} -(t-t_{i+1}) \rho^{i}}{ \tau},\hbox{ a.e. in }\Omega, \quad \hbox{ for any }t\in [t_i,t_i+1),\quad i=0,.... n-1.$$ 
 In particular one sees that 
\begin{equation}\label{tilderhoEst}
	 \tilde \rho^k(t)-\rho^\tau(t)= (t-t_i)\partial_t \rho^\tau(t),\hbox{ a.e. in }\Omega, \quad \hbox{ for any }t\in [t_i,t_{i+1}),\: i=0,1,... n-1,  
\end{equation}
and for a.e. $t\in (0,T),$  the triplet $(\tilde \rho^\tau,\rho^\tau,p^\tau)$ satisfies the following PDE  (in the sense of Corollary \ref{CorEst}) 
 	\begin{equation}
 	\label{PDEtau}
 	\left\{\begin{array}{ll} 
 		\left.  \begin{array}{ll} 
 			\partial_t 	\tilde \rho_k^\tau -  \nabla\cdot (\sigma_k\:  \nabla \eta _k^ \tau+\rho^{i-1}_kV_k^\tau +\overline f_{k}^\tau ) =   f_{0k}^\tau   ,    \\ 
 			\\  			   \eta_k^\tau \in \partial\beta(.,\rho_1^\tau+\rho_2^\tau)+\partial I\!\!I_{[0,\infty)}(\rho_k^\tau)    
 		\end{array}\right\} 	\quad & \hbox{ in } \Omega, \\  \\
 		(\sigma_k\: \nabla\eta_k^\tau+\rho^{i-1}_kV_k^\tau+\overline f_{k}^\tau  ) \cdot \nu  =  \pi_k   & \hbox{ on }\Gamma_{N _k}\\  \\
 		\eta_k^\tau  =  g _k ,\quad    & \hbox{ on }\Gamma_{D _k}
 	\end{array}
 	\right\} \quad  \hbox{ for }k=1,2, 	 
 \end{equation} 
  The next step is to analyze the behavior of the approximation  triplet $(\rho^\tau,\tilde \rho^\tau,\eta^\tau)$   as the parameter $\tau\to 0.$ We aim to proof that,   selecting a subsequence if necessary, the limit of this triplet converges to a solution of the system of PDEs given by equation \eqref{PDEevol}. This objective motivates the following proposition, whose proof is deferred to the end of this section.
 
\begin{proposition}\label{Pcarcterisation}
	There exists   $\rho=(\rho_1,\rho_2)\in L^\infty\left(0,T;[L^2(\Omega)]^2\right)\cap W^{1,2}(0,T; [H^{-1}_{\Gamma_D}(\Omega)]^2)$,  $\eta=(\eta_1,\eta_2)\in L^2\left(0,T; [H^1(\Omega)]^2 \right),$  such that $\eta=g$ on $\Gamma_D,$ and there exists sub-sequences that we denote by again $(\rho^\tau_1, \rho^\tau_2)$, $(\tilde \rho^\tau_1, \tilde \rho^\tau_2)$  and $(\eta^\tau_1,\eta^\tau_1),$  such that for each $k=1,2,$ 
	\begin{equation}\label{limrhotau}
		\rho^\tau_k \to \rho_k,\quad \hbox{ in }  [L^2(Q)]^2-\hbox{weak}, 
	\end{equation} 
		\begin{equation}\label{limrhotildetau}
		\tilde \rho^\tau_k \to \rho_k,\quad \hbox{ in }  [L^2(Q)]^2-\hbox{weak}, 
	\end{equation} 
\begin{equation}\label{limrhotildetaut}
\partial_t\tilde \rho^\tau_k \to \partial_t \rho_k,\quad \hbox{ in }  L^2(0,T; [H^{-1}_{\Gamma_D}(\Omega)]^2) -\hbox{weak}, 
\end{equation} 
and 
	\begin{equation}\label{limetatau}
		\eta^\tau_k \to \eta_k,\quad \hbox{ in } L^2\left(0,T; [H^1(\Omega)]^2  \right)-\hbox{weak}.
	\end{equation} 
	Moreover,     $(\rho,\eta)$ satisfies the state equations \eqref{Stateequations}, and  $\rho_k(0)=\rho_{0k}.$ 
\end{proposition}
\medskip  		
Then, the proof of Theorem \ref{texistuniq} follows simply by passing to the limit in the weak formulation of \eqref{PDEtau}. 
\begin{proofth}{Proof of Theorem \ref{texistuniq}}
	Thanks to Proposition \ref{Pcarcterisation}, it remains to proof that the couple $(\rho,\eta)$ satisfies the weak formulation \eqref{weakform}. 	Remember that, for any $t\in (0,T)$ and $\xi\in [H^{1}_{\Gamma_D}(\Omega)]^2,$ we have  
	$$  \int _\Omega 	\partial_t 	\tilde \rho ^\tau_k (t)\cdot  \xi +\int_\Omega  \sigma_k \: \nabla\eta ^\tau_k  (t) \cdot \nabla \xi_k  = \langle f  ^\tau_k(t) ,\xi_k\rangle _{k,\Omega} - \int_\Omega   (\rho^\tau_k (t-\tau)V^\tau_k (t)\cdot \nabla \xi_k   + \langle \pi_k ,\xi_k \rangle _{\Gamma_{N_k}}  ,$$
	and then 
	\begin{equation}\label{eqtau1}
		\begin{array}{c} 	  \frac{d}{dt} \int _\Omega  	\tilde \rho ^\tau_k  (t) \cdot \xi_k  +\int_\Omega  (\sigma_k \:  \nabla\eta   ^\tau_k  (t) +  \rho^\tau_k (t) V^\tau_k (t)) \cdot \nabla \xi_k   = \langle f  ^\tau_k(t) ,\xi_k\rangle _{k,\Omega}   + \langle \pi_k ,\xi_k \rangle _{\Gamma_{N_k}}   \\ \\ + \underbrace{\int_\Omega  (\rho^\tau_k(t) - \rho^\tau_k(t-\tau) )V^\tau_k(t) \cdot \nabla \xi _k]}_{J_\tau(t)} ,\quad \hbox{ in }\D'([0,T))  .
		\end{array} 
	\end{equation}   
	Using the fact that  $\rho^\tau_k$ is bounded in $ L^\infty(0,T;L^2(\Omega)) $ and $V^\tau_k$ is relatively compact in $ L^2(Q)^N ,$ we see that  $J_\tau\to 0$ weakly. 
	So, passing to the limit in \eqref{eqtau1}, we deduce that   the couple $(\rho,\eta )$  satisfies \eqref{weakform}. Thus, $(\rho,\eta)$ is a weak solution of the problem \eqref{weakform}. 
	
\end{proofth}

\bigskip 
Now, let us prove Proposition \ref{Pcarcterisation}. We begin by to prove the following result.

 \begin{lemma}\label{Lrhotproperty}
 	For any $t\in [0,T) ,$ we have 
 	\begin{equation}\label{rhotproperty}
 		\begin{array}{c}
 		 	\int_\Omega ( \beta[\rho^{\tau}(t)]  -   \rho^{\tau}(t)\cdot  \tilde g  )  +   \int_0^t \!\!    \int_\Omega  F_\sigma[ \nabla(\eta^\tau -\tilde g) ]    \leq 	\int_\Omega ( \beta[\rho_0]  -   \rho_0\cdot  \tilde g  )  +  \int_0^t 	 [ f^{\tau},    \eta^{\tau}-\tilde g ]_{\Omega}  \\  \\  
 			\hspace*{3cm}	-    \int_0^t \!\! \int_\Omega  [\rho^{\tau}(t-\tau) V^{\tau}   \cdot \nabla(\eta^{\tau}-\tilde g)]  +T\:  \int_\Omega  F_\sigma[\nabla  \tilde g]      +  \int_0^t   [ \pi,\eta^{\tau}]_{\Gamma_N},\quad \hbox{ for any  } t\in [0,T)  .  
 		\end{array}
 	\end{equation}	   		 
 \end{lemma}
 \begin{proof} 
Thanks to \eqref{eqextreme}, we know that 
 \begin{equation}   \label{eqextremei}
 	\begin{array}{c}
 		\int_\Omega  \rho^{i} \cdot (\eta^{i}-\tilde g)+  \tau\: \int_\Omega  [\sigma\: \nabla\eta^{i}\cdot \nabla (\eta^{i} -\tilde g)]  
 		=  	\int_\Omega  \rho^{i-1} \cdot (\eta^{i}-\tilde g) +	\tau \: [  f^{i},\eta^{i}-\tilde g]_{\Omega} \\  \\  -\tau  \int_\Omega [\rho^{i-1}\: V^{i}\cdot \nabla(\eta^{i}-\tilde g)]     + \tau [ \pi,\eta^{i}]_{\Gamma_N}     .    
 	\end{array} 
 \end{equation}   	Remember that $\ \eta_1^{i} \vee \eta_2^{i}  \in \partial \beta(.,\rho_1^{i}+\rho_2^{i}).$  This implies that   $(\rho^{i}-\rho^{i-1})\cdot \eta^{i}\geq \beta[\rho^{i}]-  \beta[\rho^{i-1}],$ and then 
 	\begin{equation}\label{eqextreme2}
 		\begin{array}{c}
 			\int_\Omega  ( \beta[\rho^{i}] - 	  \rho^{i} \cdot \tilde g ) -  	\int_\Omega  ( \beta[\rho^{i-1}]- 	  \rho^{i-1} \cdot  \tilde g ) + \tau\:   \int_\Omega  F_\sigma [\nabla(\eta^i-\tilde g) ]   \leq     \tau \: [  f^{i},\eta^{i}-\tilde g]_{\Omega} \\  \\  - \tau \:  \int_\Omega [\rho^{i-1}\: V^{i}\cdot \nabla(\eta^{i}-\tilde g)]    + \tau  \int_\Omega  F_\sigma[\sigma\:\nabla  \tilde g]     + \tau \:  [ \pi,\eta^{i}]_{\Gamma_N}   .	\end{array} 
 	\end{equation}	
   	  Adding this inequality for $i=1...l\: \leq n,$  we obtain    
   	 	\begin{equation} 
   	  	\begin{array}{c}
   	  		\int_\Omega ( \beta[\rho^{\tau}(t_l)]  -   \rho^{\tau}(t_l)\cdot  \tilde g  )  +    \sum_{i=1}^{l}   \int_{t_{i-1}}^{t_i}  \!\!    \int_\Omega  F_\sigma[\nabla(\eta^\tau -\tilde g) ]    \leq 	\int_\Omega ( \beta[\rho_0]  -   \rho_0\cdot  \tilde g  )  +  \sum_{i=1}^{l}   \int_{t_{i-1}}^{t_i}    [ f^{\tau},    \eta^{\tau}-\tilde g ]_{\Omega}  \\  \\  
   	  	 	-    \sum_{i=1}^{l}   \int_{t_{i-1}}^{t_i}   \!\! \int_\Omega  [\rho^{\tau}(t-\tau) V^{\tau}   \cdot \nabla(\eta^{\tau}-\tilde g)]  + T\:  \int_\Omega  F_\sigma[\sigma\:\nabla  \tilde g]     +  \sum_{i=1}^{l}   \int_{t_{i-1}}^{t_i}    [ \pi,\eta^{\tau}]_{\Gamma_N},\quad \hbox{ in }\D'([0,T) , 
   	  	\end{array}
   	  \end{equation}
   	  which implies
   	  	\begin{equation} 
   	  	\begin{array}{c}
   	  		\int_\Omega ( \beta[\rho^{\tau}(t_l)]  -   \rho^{\tau}(t_l)\cdot  \tilde g  )  +   \int_{0}^{t_l}  \!\!    \int_\Omega  F_\sigma[\nabla(\eta^\tau -\tilde g) ]    \leq 	\int_\Omega ( \beta[\rho_0]  -   \rho_0\cdot  \tilde g  )  +    \int_{0}^{t_l}     [ f^{\tau},    \eta^{\tau}-\tilde g ]_{\Omega}  \\  \\  
   	  	 	-   \int_{0}^{t_l}    \!\! \int_\Omega  [\rho^{\tau}(t-\tau) V^{\tau}   \cdot \nabla(\eta^{\tau}-\tilde g)]  + T\:  \int_\Omega  F_\sigma[\nabla  \tilde g]     +   \int_{0}^{t_l}     [ \pi,\eta^{\tau}]_{\Gamma_N},\quad \hbox{ in }\D'([0,T) .
   	  	\end{array}
   	  \end{equation} 
   	Since $l$ is arbitrary, we deduce     the result of the lemma. 
 
 \end{proof}
 
 \begin{lemma}\label{Propcompacttau}
  There exists a constant $C=C(p,N,\Omega)$ such that :  
 		\begin{equation}\label{Estpple}
 			\begin{array}{c}
 				\int_\Omega  (  \rho_1^{\tau}(t)+  \rho_2^{\tau}(t)  -M)^{+2}  +\int_0^T\!\! \int_\Omega   F_\sigma[\nabla(\eta^{\tau}-\tilde g)]  \\  \\  \leq   	\frac{B}{C}\left\{  1+      \frac{  T   \Vert V^{\tau}\Vert_\infty^2}{C} \:  \exp\left( \frac{T  \Vert V^{\tau}\Vert_\infty^2}{C}\right)  \right\} ,
 			\end{array}  
 		\end{equation}	 
 		where $B$ is given by 
 		$$	 
 		B=  \Vert \tilde g\Vert_{[L^2(\Omega)]^2}^2+   \int_\Omega  F_\sigma[\nabla  \tilde g]    +  \int_\Omega ( \beta[\rho_0]  -   \rho_0\cdot  \tilde g  )   	     
 		+    T\: \int_\Omega  (\rho_{01}+\rho_{02}-M)^{+2}  $$
 		$$  +    \sup_{0<\tau<1}   \Vert V^{\tau} \Vert_\infty^2  +  \int_0^T\!\!   \Vert \pi\Vert_{[H^{-1/2}_{\Gamma_{N}}]^2}^2    +  \sup_{0<\tau<1}   \int_0^T\!\!  \left(   \Vert   f_{0}^\tau \Vert_{[L^2(\Omega)]^2}^2  + \Vert \overline f^\tau  \Vert_{[L^2(\Omega)]^2}^2 \right)  $$
 \end{lemma}
 \begin{proof} 
 	Thanks to  Lemma \ref{Lrhotproperty},   for any  $t\in [0,T) ,$  we have 
 	\begin{equation} \label{esttau}
 		\begin{array}{c}
 			\int_\Omega ( \beta[\rho^{\tau}(t)]  -   \rho^{\tau}(t)\cdot  \tilde g  )+ \int_0^t\!\! \int_\Omega    F_\sigma[\nabla(\eta^{\tau} -\tilde g)]\leq 	  \int_\Omega ( \beta[\rho_0]  -   \rho_0\cdot  \tilde g  )  +   T\:  \int_\Omega  F_\sigma[\sigma\:\nabla  \tilde g]  	  \\  \\  +  \underbrace{\int_0^t\!\! \int_\Omega\vert   [\rho^{\tau}(.-\tau) V^{\tau}   \cdot \nabla(\eta^{\tau}-\tilde g)]  \vert }_{J_1^\tau }    +    \underbrace{ \int_0^t  \vert [ \pi,\eta^{\tau}]_{\Gamma_N}	  +  [ f^{\tau},    \eta^{\tau}-\tilde g ]_{\Omega}  \vert  }_{J_2^\tau } .
 		\end{array}
 	\end{equation}	   
 Using Young inequality with $\epsilon$ and $C_\epsilon,$ we see that 
 	\begin{eqnarray*}
J_1 ^\tau &\leq&   \int_0^t\!\! \int_\Omega  \vert   \rho^{\tau}_1(.-\tau) V^{\tau}_1   \cdot \nabla(\eta^{\tau}_1-\tilde g_1)   +    \rho^{\tau}_2(.-\tau) V^{\tau}_2  \cdot \nabla(\eta^{\tau}_2-\tilde g_2) \vert   \\  \\   & \leq & 	  \int_0^t\!\! \int_\Omega    (\rho^{\tau}_1(.-\tau)  +    \rho^{\tau}_2(.-\tau) )\: \left( \vert V^{\tau}_1\vert  \vee  \vert V^{\tau}_2  \vert \right) \:  ( \vert \nabla(\eta^{\tau}_1-\tilde g_1) \vert \vee 
 		\vert \nabla(\eta^{\tau}_2-\tilde g_2) \vert )   \\  \\    & \leq &  C_\epsilon  \: \frac{\Vert  V^{\tau} \Vert_\infty^2}{2 } 		  \int_0^t\!\!   \int_\Omega  (\rho_1(.-\tau)+\rho_2(.-\tau))^2 +   \epsilon	   \int_0^t\!\!   \int_\Omega   F_\sigma  [\nabla(\eta^{\tau}-\tilde g)]  .
 		 	\end{eqnarray*}
 		 	Moreover, 	since $\frac{1}{2}\vert a\vert^2 \leq  (\vert a\vert-M)^{+2}+ M^2,$  we get 
 		 	 	\begin{eqnarray*}  J_1^\tau  
 		&\leq&    C_\epsilon  \Vert V^{\tau} \Vert_\infty^2 		  \int_0^t\!\!   \int_\Omega  (\rho_1(.-\tau)+\rho_2(.-\tau)-M)^{+2} +    \epsilon	   \int_0^t\!\!   \int_\Omega   F_\sigma  [\nabla(\eta^{\tau}-\tilde g)]     \\  \\ & &   + C_\epsilon  M^2 \: T \:  \Vert V^{\tau} \Vert_\infty^2\:   \vert \Omega\vert\\  \\
 	&	\leq &    C_\epsilon  \:  \Vert V^{\tau} \Vert_\infty^2 		  \int_{-\tau}^{t-\tau }\!\!   \int_\Omega  (\rho_1+\rho_2-M)^{+2} +   \epsilon	   \int_0^t\!\!   \int_\Omega   F_\sigma  [\nabla(\eta^{\tau}-\tilde g)]        + C_\epsilon \: M^2 \: T \:  \Vert V^{\tau} \Vert_\infty^2  \vert \Omega\vert\\  \\
 		&\leq &   C_\epsilon  \:  \Vert V^{\tau} \Vert_\infty^2 		  \int_{0}^{t }\!\!   \int_\Omega  (\rho_1+\rho_2-M)^{+2} +   \epsilon	   \int_0^t\!\!   \int_\Omega   F_\sigma  [\nabla(\eta^{\tau}-\tilde g)]  + C_\epsilon \: M^2 \: T \:  \Vert V^{\tau} \Vert_\infty^2  \vert \Omega\vert    \\ \\ &  & +  \tau\: \int_\Omega  (\rho_{01}+\rho_{02}-M)^{+2}   .
 	\end{eqnarray*}
 		It is not difficult to see also that, for any $\epsilon >0,$ we can find $C_\epsilon >0,$ such that 
 	$$J_2^\tau   \leq C_\epsilon\:  \int_0^T\!\!  \left(   \Vert \pi\Vert_{[H^{-1/2}_{\Gamma_{N}}]^2}^2   +\Vert   f_{0}^\tau  \Vert_{[L^2(\Omega)]^2}^2  + \Vert \overline f ^\tau \Vert_{[L^2(\Omega)]^2}^2 \right)  +    \epsilon	   \int_0^t\!\!   \int_\Omega   + C_\epsilon \: M^2 \: T \:  \Vert V^{\tau} \Vert_\infty^2  \vert \Omega\vert    [\nabla(\eta^{\tau}-\tilde g)]  . $$
 Using moreover, the fact that    $$  \beta[\rho^{\tau}]  -   \rho^{\tau}\cdot  \tilde g   \geq \frac{C}{2} (\rho^{\tau}_1+\rho^{\tau}_2 -M)^{+2} -\frac{1}{2C} \vert \tilde g_1\vee \tilde g_2\vert^2 \geq  \frac{C}{2} (\rho^{\tau}_1+ \rho^{\tau}_2 -M)^{+2} -\frac{1}{2C} \vert \tilde g \vert^2,$$ 
 \eqref{esttau}  implies that  
 	\begin{equation}  
 		\begin{array}{c}
 			\frac{C}{2} \int_\Omega  (\rho^{\tau}_1(t)+\rho^{\tau}_2(t) -M)^{+2} + (1-2\epsilon) \int_0^t\!\! \int_\Omega    F_\sigma[\nabla(\eta^{\tau} -\tilde g)]\leq 	 \frac{1}{2C} \Vert \tilde g\Vert_{[L^2(\Omega)]^2}^2  
 			+  T \int_\Omega  F_\sigma[\nabla  \tilde g]  \\  \\ 
 			+	 \int_\Omega ( \beta[\rho_0]  -   \rho_0\cdot  \tilde g  )       +   \tau\: \int_\Omega  (\rho_{01}+\rho_{02}-M)^{+2}    +  C_\epsilon  \:  \Vert V^{\tau} \Vert_\infty^2 		  \int_{0}^{t }\!\!   \int_\Omega  (\rho_1+\rho_2-M)^{+2}  \\  \\ 
 			+  \int_0^T\!\!  \left(   \Vert \pi\Vert_{[H^{-1/2}_{\Gamma_{N}}]^2}^2   +\Vert   f_{0}^\tau  \Vert_{[L^2(\Omega)]^2}^2  + \Vert \overline f ^\tau \Vert_{[L^2(\Omega)]^2}^2 \right) 
 			+ C_\epsilon \: M^2 \: T \:  \Vert V^{\tau} \Vert_\infty^2  \vert \Omega\vert   ,\quad \hbox{ for any }t\in [0,T)  .  
 		\end{array}
 	\end{equation}	 
 	Working with a fixed small $\epsilon,$  we can find  a constant   $C=C(p,N,T,\Omega )>0,$ such that   
 	\begin{equation}\label{Esttau1}
 		\begin{array}{c}
 		 \int_\Omega  (\rho^{\tau}_1(t)+ \rho^{\tau}_2(t) -M)^{+2}  +\int_0^t\!\! \int_\Omega   F_\sigma[\nabla\eta^{\tau}-\tilde g ]  \\  \\  \leq   	\frac{B}{C}+     \frac{  \Vert V^{\tau} \Vert_\infty^2 }{C} \: 	 \int_0^t\!\!   \int_\Omega   (\rho^{\tau}_1(t)+ \rho^{\tau}_2(t) -M)^{+2}     .
 		\end{array}  
 	\end{equation}	 	    
 	This implies that  
 	\begin{equation}
 	 \int_\Omega  (\rho^{\tau}_1(t)+ \rho^{\tau}_2(t) -M)^{+2}    \leq  \frac{B}{C} \: \exp\left( T  \Vert V^{\tau}\Vert_\infty^2/C\right) ,
 	\end{equation}
  and   \eqref{Estpple} follows by \eqref{Esttau1}.  
 	 \end{proof}

 \begin{proofth}{Proof of Proposition \ref{Pcarcterisation}} Thanks to \eqref{Estpple} it is clear that  $\rho^\tau$ and $\tilde \rho^\tau$ are    bounded in $L^\infty\left(0,T; [L^2(\Omega)]^2 \right),$     and $\eta^\tau$   is bounded in  $L^2\left(0,T; [H^1(\Omega)]^2  \right).$ This implies  \eqref{limrhotau} and   \eqref{limetatau}.    On the other hand,  since 
 	$$	\partial_t 	\tilde \rho_k ^\tau =     f_{k0} ^\tau +\nabla \cdot (\rho_k^\tau(.-\tau)V_k^{\tau} + \sigma_k\:  \nabla\eta_k ^\tau+   \overline f_k^\tau)  \quad \hbox{ in }\Omega, $$
 	with   $  ( \sigma_k\: \nabla\eta_k^\tau+\rho^\tau(.-\tau)V_k^{\tau}  +\overline f_k^\tau)  \cdot \nu  =  \pi_k $  on  $\Sigma_{N_k  },$ and by   \eqref{Estpple},  $ \sigma_k\:  \nabla\eta_k    ^\tau+\rho_k^\tau(.-\tau)V_k^{\tau}  +\overline f_k^\tau $ is bounded in $L^2(Q),$ we see that  $	\partial_t 	\tilde \rho_k^\tau $ is bounded in $L^2\left(0,T;H^{-1}_{\Gamma_{D_k}}(\Omega) \right).$    	Combining this with \eqref{tilderhoEst} we deduce that $\tilde \rho^\tau$ and $\rho^\tau$ have the same $[L^2(Q)]-$weak limit. Thus   \eqref{limrhotildetau} and   \eqref{limrhotildetaut}.    Let us prove now that the couple $(\rho,\eta)$ satisfies the state equations \eqref{Stateequations}.  Thanks again to \eqref{Estpple},  we see that  $\eta_1^\tau\vee \eta_2^\tau$ is bounded in $L^2\left(0,T; H^1(\Omega)\right),$ and      there exists  $\overline \eta \in L^2\left(0,T; [H^1 (\Omega)]^2 \right) $ such that
 	 \begin{equation}\label{ineqeta}
 		0\leq {\eta_1}\vee \eta_2  \leq \overline \eta ,\quad \hbox{ a.e. in }Q, 
 	\end{equation}  
 	and,   by taking a sub-sequence if necessary,  
 	\begin{equation}
 		\eta_1^\tau\vee \eta_2^\tau \to \overline \eta ,\quad \hbox{ in } L^2\left(0,T;  H^{1} (\Omega) \right)-\hbox{weak},
 	\end{equation}  
Now,  using  weak compensated compactness results (cf. \cite{Moussa,ACM}),   we get 
 	\begin{equation} 
 		\int_0^T\!\!\int_\Omega  \tilde  \rho_k^\tau \: \eta_k^\tau \: \varphi \quad \xrightarrow[\tau \to 0 ]{}  	\quad \int_0^T\!\!\int_\Omega \rho_k\: \eta_k \: \varphi,  \hbox{ and }k=1,2, 
 	\end{equation}  
and  also   
  		\begin{equation} 
  		\int_0^T\!\!\int_\Omega    (\tilde \rho_1^\tau +\tilde \rho_2^\tau) \: 	\eta_1^\tau\vee \eta_2^\tau  \: \varphi \quad \xrightarrow[\tau \to 0 ]{}  	\quad 	\int_0^T\!\!\int_\Omega (\rho_1+\rho_2)\: \overline  \eta    \: \varphi ,
  	\end{equation} 
  	 for any $\varphi\in \D(Q).$   	Combining this with  \eqref{tilderhoEst} and the fact that $\partial_t\tilde \rho^\tau$ and $\eta^\tau$ are   bounded in 
  	 $L^2\left(0,T;H^{-1}_{\Gamma_{D_k}}(\Omega) \right)$     and 
  	$L^2\left(0,T; H^1(\Omega)\right),$  respectively, we deduce that 
  	 	\begin{equation}\label{convetak}
  	 	\int_0^T\!\!\int_\Omega     \rho_k^\tau \: \eta_k^\tau \: \varphi \quad \xrightarrow[\tau \to 0 ]{}  	\quad \int_0^T\!\!\int_\Omega \rho_k\: \eta_k \: \varphi,  \hbox{ and }k=1,2, 
  	 \end{equation}  
  	 and      
  	 \begin{equation}\label{convetak2}
  	 	\int_0^T\!\!\int_\Omega    (  \rho_1^\tau +  \rho_2^\tau) \: 	\eta_1^\tau\vee \eta_2^\tau  \: \varphi \quad \xrightarrow[\tau \to 0 ]{}  	\quad 	\int_0^T\!\!\int_\Omega (\rho_1+\rho_2)\: \overline  \eta    \: \varphi ,
  	 \end{equation} 
  	 for any $\varphi\in \D(Q).$    
     Moreover, since  $\eta_1^\tau\vee \eta_2^\tau   \in \partial \beta( .,   \rho_1^\tau +\rho_2^\tau), $ a.e. in $Q,$  by using  Weak Aubin's type Lemma  (cf.  Proposition 1.4 in \cite{ACM}), we see that  $$\overline \eta  \in    \partial \beta( .,  \rho_1 + \rho_2),\quad \hbox{ a.e. in }Q.$$   To finish the proof, let us  justify that ${\eta_1}\vee \eta_2       \in    \partial \beta( .,   \rho_1 + \rho_2),$ a.e. in $Q.$  First, one sees that since   $ (\rho_1^\tau +\rho_2^\tau)\:  \eta_1^\tau\vee \eta_2^\tau=\rho_1^\tau \: \eta_1^\tau+ \rho_2^\tau \: \eta_2^\tau,$   \eqref{convetak} and \eqref{convetak2} imply 
     \begin{equation}\label{carp}
     	\rho_1 \: \eta_1+ \rho_2 \: \eta_2 = (\rho_1 +\rho_2)  \:\overline  \eta   ,\quad \hbox{ a.e. in }Q, 
     \end{equation} 
     and, using \eqref{ineqeta},  we deduce that 
     \begin{equation}\label{carp1}
     	\rho_1 \: \eta_1+ \rho_2 \: \eta_2 = (\rho_1 +\rho_2)  \:   {\eta_1}\vee \eta_2   ,\quad \hbox{ a.e. in }Q. 
     \end{equation}  
     So, if $\rho_1+\rho_2 \neq 0,$ \eqref{carp} and  \eqref{carp1}  imply   clearly that  $ {\eta_1}\vee \eta_2   =\overline \eta \in  \partial \beta( .,  \rho_1 +\rho_2),$ a.e. in $Q.$   And, if     $\rho_1+\rho_2 = 0,$ the result follows from the fact that   $ \partial \beta(.,0)\ni 0\leq {\eta_1}\vee \eta_2  \leq  \overline \eta\in  \partial \beta(.,0) .$

 \end{proofth}

  \begin{remark}
 Thanks to compensated compactness results of  Proposition 1.4 in \cite{ACM},   if  $\beta^{-1}(x,.)$      is single valued then we have $\rho_1^\tau + \rho_2^\tau \to \rho_1+\rho_2$   in $L^2(Q).$   However, it is not clear if this remains to be true for each density $\rho_k.$ 
  	\end{remark}

\section{Connection with $H^{-1}$-like theory in the homogeneous case}
    \label{SH-1}

 The aim of this section is to prove Theorem \ref{ThH-1}. So,  we assume that 
$$\pi=0  \hbox{ on }\Gamma_N   \quad  \hbox{ and } \quad  g=0   \hbox{ on }\Gamma_D .$$
   Let us consider the functional $\E$ as defined in Section \ref{Smain}.   Then, we consider  the Hilbert space  
$[H^{-1}_{\Gamma_{D}}(\Omega)]^2,$ endowed with the norm 
$$\Vert f\Vert_{H^{-1}_{\Gamma_{D}}(\Omega)}  =\left( \Vert f_1\Vert _{H^{-1}_{\Gamma_{D_1}}}^2  +\Vert f_2\Vert _{H^{-1}_{\Gamma_{D_2}}} ^2   \right)^{1/2}, \quad \hbox{ for any }f=(f_1,f_2)\in  [H^{-1}_{\Gamma_{D}}(\Omega)]^2, $$
and its  associate  inner product  
$$[ f,g]_{\Omega} =  \langle f_1,g_1\rangle_{1,\Omega}  + \langle f_2,g_2\rangle_{2,\Omega} ,\quad \hbox{ for any }f,g\in   [H^{-1}_{\Gamma_{D}}(\Omega)]^2.$$

\bigskip 
 Our   main result  in this section concerns  the   characterization of $\partial E$ in terms of cross-diffusion system like in \eqref{PDE1}.  As we'll see this operator is closely connected to the stationary problem 
 \begin{equation}  \label{PDEstat}                                                                                                                 
 	\left\{\begin{array}{ll}                                                                                                
 		\left.  \begin{array}{ll}                                                                                                      
 			- \nabla \cdot(\sigma_k\:  \nabla \eta_k+ \overline f_k) = f_{0 k}   ,  \quad  \hbox{ for }k=1,2 \\               
 			\\  			    \eta_k \in \partial\beta(.,\rho_1+\rho_2)+\partial I\!\!I_{[0,\infty)}(\rho_k)  \\    	 
 		\end{array}\right\} 	\quad & \hbox{ in } \Omega, \\  \\                                                                        
 		(\sigma_k\:\nabla\eta_k+\overline f_k)  \cdot \nu  = 0 & \hbox{ on }\Gamma_{N _k}\\  \\                                        
 		\eta_k  =  0 \quad    & \hbox{ on }\Gamma_{D_k},                                                                            
 	\end{array}  	\right. 	\quad k=1,2,                                                                                                                       
 \end{equation}           
for a given $f=(f_1,f_2) \in H^{-1}_{\Gamma_{D_1}}(\Omega)\times  H^{-1}_{\Gamma_{D_2}}(\Omega) ,$  and for each $k=1,2,$ $(f_{0k},\overline f_k)\in L^2(\Omega)\times L^2(\Omega)^N$ is the couple associated with $f_k\in H^{-1}_{\Gamma_{D_k}}(\Omega),$ $f_k=f_{0k}+\nabla \cdot \overline f_k.$

  \begin{proposition}\label{PH-1}
  	For any $(\rho,f)\in [H^{-1}_{\Gamma_{D}}(\Omega)]^2 \times [H^{-1}_{\Gamma_{D}}(\Omega)]^2,$  $f\in \partial E(\rho)$ if and only if  $\rho:=(\rho_1,\rho_2)\in [L^2_+(\Omega)]^2, $  and for each $k=1,2,$ there exists  $\eta_k\in H^1 _{\Gamma_{D_k}}(\Omega),$  such that 
  		$$ \eta_1\vee\eta_2 =:\tilde\eta \in \partial \beta(x,   \rho_1+ \rho_2),\: \quad  \eta_k -\tilde \eta    \in  \partial  I\!\! I_{[0,\infty)}(\rho_k),   \quad \hbox{ a.e.  in }\Omega,$$ 
  		and 
  		$$  \int_\Omega\sigma_k\:\nabla\eta_k \cdot \nabla \xi_k  =\langle f_k ,\xi_k\rangle_{k,\Omega}     -\langle \pi_k, \xi_k \rangle_{\Gamma_k} ,\quad  \hbox{ for any }\xi_k\in  H^1_{\Gamma_{_k} } (\Omega).$$

  	\end{proposition}
  \begin{proof}    Our aim is to prove that  
  	\begin{equation}\label{identificationA}
  	 \partial \E=	\A ,
  	\end{equation}
  	where $\A$ is the operator define in $[H^{-1}_{\Gamma_{D}}(\Omega)]^2,$ by 
  	$f\in \A(\rho)$ if and only if $f\in  [H^{-1}_{\Gamma_{D}}(\Omega)]^2,$  $\rho\in [L^2(\Omega)_+]^2$ and $\rho$ is a solution of \eqref{PDEstat}, in the sense of Proposition \ref{PH-1}.  
  	
  		Thanks to classical theory of maximal monotone graphs (cf. \cite{Br}), since $\E$ is   convex and l.s.c.   in $[H^{-1}_{\Gamma_{D}}(\Omega)]^2,$ the operator $\partial E$ defines a maximal monotone graph in $[H^{-1}_{\Gamma_{D}}(\Omega)]^2.$   Since E is both convex and lower semi-continuous in $[H^{-1}_{\Gamma_{D}}(\Omega)]^2,$ $\partial E$  itself defines a maximal monotone graph within the same space. Furthermore, Corollary \ref{CorEst} guarantees that for any $f\in [H^{-1}_{\Gamma_{D}}(\Omega)]^2,$  there exists $\rho\in  [L^2_+(\Omega)]^2,$   such that $f=\rho+\A(\rho).$   
  	 	To establish the identification in equation (\ref{identificationA}), it suffices to prove  that $\A\subset \partial \E$. Indeed,  proving this inclusion, implies that  $\A$ is a maximal monotone graph contained within $\partial \E.$  Consequently, the two graphs must coincide.
  	  To this aim, we  consider $f\in  [H^{-1}_{\Gamma_{D}}(\Omega)]^2 $  and $\rho\in [L^2(\Omega)_+]^2$ be such that $f\in \A(\rho),$ and we  prove that  
\begin{equation}\label{eqsub}
	  	   [f,z-\rho]_{\Omega} \leq \int_\Omega \beta[z] - \int_\Omega \beta[\rho]    ,\quad \hbox{ for any }z\in  [L^2(\Omega)_+]^2. 
\end{equation}
Let us consider $\phi_z$  and $\phi_\rho $ given by 
$$\phi^z_k =  \argmin_{\omega \in L^2(\Omega)^N}\left \{   \int_\Omega    \sigma_k\:\vert  \omega   \vert^2    \:  :\:  - \nabla \cdot (\sigma_k\:\omega)   =z_{k}  \hbox{ in }\Omega \hbox{ and }  \sigma_k\: \omega  \cdot \nu =0 \hbox{ on }\Gamma_{N_k }  \right\}  , $$   and 
$$\phi^\rho_k =  \argmin_{\omega \in L^2(\Omega)^N}\left \{   \int_\Omega  \sigma_k\: \vert \omega   \vert^2    \:  :\:  - \nabla \cdot (\sigma_k\:\omega)   =\rho_{k}  \hbox{ in }\Omega \hbox{ and }    \sigma_k\: \omega  \cdot \nu =0 \hbox{ on }\Gamma_{N_k }  \right\} .$$
Since $f_k$ satisfies \eqref{PDEstat},  using the definition of $\langle .,.\rangle_{k,\Omega}$, we have 
\begin{eqnarray*}
 [f,z-\rho]_{\Omega}  &=& \langle f_1,z_1-\rho_1\rangle_\Omega +  \langle f_2,z_2-\rho_2\rangle_\Omega \\  \\
&=& \int_\Omega \sigma_1 \: \nabla \eta_1\cdot( \phi^z_1 -  \phi^\rho_1)+ \int_\Omega \sigma_2 \: \nabla \eta_2\cdot( \phi^z_2 -  \phi^\rho_2)\\ \\ 
&=& \int_\Omega  \eta_1\: (z_1 - \rho_1)  + \int_\Omega   \eta_2\:(z_2- \rho_2) \\ \\ 
&=& \int_\Omega (\eta_1\: z_1 +   \eta_2\: z_2)  -  \int_\Omega (\eta_1\: \rho_1 +   \eta_2\: \rho_2).
\end{eqnarray*}
Remember that $(\eta_1,\eta_2)\in \partial \beta(.,(\rho_1,\rho_2)),$ a.e. in $Q.$  So $\eta_1\rho_1+\rho_2\eta_2= (\rho_1+\rho_2)\eta_1\vee \eta_2$ and  ${\eta_1}\vee \eta_2 \in \partial \beta(x,\rho_1 +    \rho_2),$ a.e. in $\Omega$. This implies that      
\begin{eqnarray*}
	[f,z-\rho]_{\Omega}
&=&    \int_\Omega (\eta_1\: z_1 +   \eta_2\: z_2)  -  \int_\Omega (\rho_1 +    \rho_2) {\eta_1}\vee \eta_2 \\ \\ 
&\leq & \int_\Omega  \int_\Omega (z_1 +    z_2) {\eta_1}\vee \eta_2    -  \int_\Omega (\rho_1 +    \rho_2) {\eta_1}\vee \eta_2 \\  \\ 
&\leq& \int_\Omega \beta[z] - \int_\Omega \beta[\rho]\quad = \E(z) -\E(\rho),
\end{eqnarray*}
 Thus \eqref{eqsub}. 
  
 	\end{proof}

\begin{proofth}{Proof of Theorem \ref{ThH-1}} The proof follows directly from Proposition \ref{PH-1} and the fact that   $h=f-[\nabla \cdot (\rho\: V)] \in L^2(0,T;[H^{-1}_{\Gamma_D}(\Omega)]^2).$  Then, the proof is a direct consequence of the definition of weak solution of   \eqref{PDEevol} and the characterization of $\partial \E$ in terms of PDE in Proposition \ref{PH-1}. 
\end{proofth}  
%
%
%

\end{document}